\documentclass[11pt]{article} 
\usepackage[utf8]{inputenc}
\usepackage[T1]{fontenc}
\usepackage[english]{babel}
\makeatletter
\newcommand{\note}[1]{{\textcolor{red}{[#1]}}\@latex@warning{Note: #1}}
\makeatother
\usepackage{amsmath,amsfonts,amssymb,amsthm,mathtools}
\usepackage[shortlabels]{enumitem}

\usepackage[mono=false]{libertine}
\usepackage[cmintegrals,libertine]{newtxmath}
\usepackage[cal=euler, scr=boondoxo]{mathalfa}

\useosf
\usepackage[a4paper,vmargin={3.5cm,3.5cm},hmargin={2.5cm,2.5cm}]{geometry}
\usepackage[font={small,sf}, labelfont={sf,bf}, margin=1cm]{caption}
\usepackage[dvipsnames]{xcolor}
\usepackage[pdftex,colorlinks=true]{hyperref}
\usepackage[pdftex]{color,graphicx}

\usepackage{stackrel}
\usepackage{soul}
\hypersetup{linkcolor=BrickRed,citecolor=NavyBlue,urlcolor=NavyBlue}

\theoremstyle{plain}
\newtheorem{theorem}{Theorem}
\newtheorem{corollary}[theorem]{Corollary}
\newtheorem{proposition}[theorem]{Proposition}
\newtheorem{lemma}[theorem]{Lemma}
\theoremstyle{definition}

\newcommand{\ind}{\mathbf{1}}

\newcommand{\Z}{\mathbb{Z}}

\newcommand{\rmd}{\mathrm{d}}

\newcommand{\map}{\mathfrak{m}}
\newcommand{\intersec}{\operatorname{i}}
\newcommand{\chords}{\mathcal{C}}
\newcommand{\maps}{\mathcal{M}}
\newcommand{\interval}[1]{[\kern-2.8pt[#1]\kern-2.8pt]}

\DeclareRobustCommand{\qbinom}{\genfrac[]{0pt}{}}

\begin{document}

\title{\bf Double-scaled SYK from boundary metrics of planar maps}

\author{\textsc{Timothy Budd}\footnote{Email: \texttt{\href{mailto:t.budd@science.ru.nl}{t.budd@science.ru.nl}}}\\
{\small IMAPP, Radboud University, Nijmegen, The Netherlands.}}
\date{\today}
\maketitle

\vspace{-6mm}
\begin{abstract}
The enumeration of planar maps with control on the boundary metric, i.e.\ the pseudometric induced on the outer face of the map by its bulk graph distance metric, is a difficult problem in general.
However, we show that for a family of bipartite planar map models with special $q$-deformed face weights that arise in the physics context of the double-scaled Sachdev-Ye-Kitaev model (DSSYK) the enumeration admits a very simple answer.
Encoding the boundary metric of a bipartite planar map by its so-called geodesic chord diagram, we prove that the weighted enumeration depends only on the crossing number of the chord diagram. 
At fixed perimeter, the induced law of the geodesic chord diagram in these planar map models coincides exactly with the chord diagram representation of the DSSYK model.
\end{abstract}

\begin{figure}[h!]
	\centering
	\includegraphics[width=\linewidth]{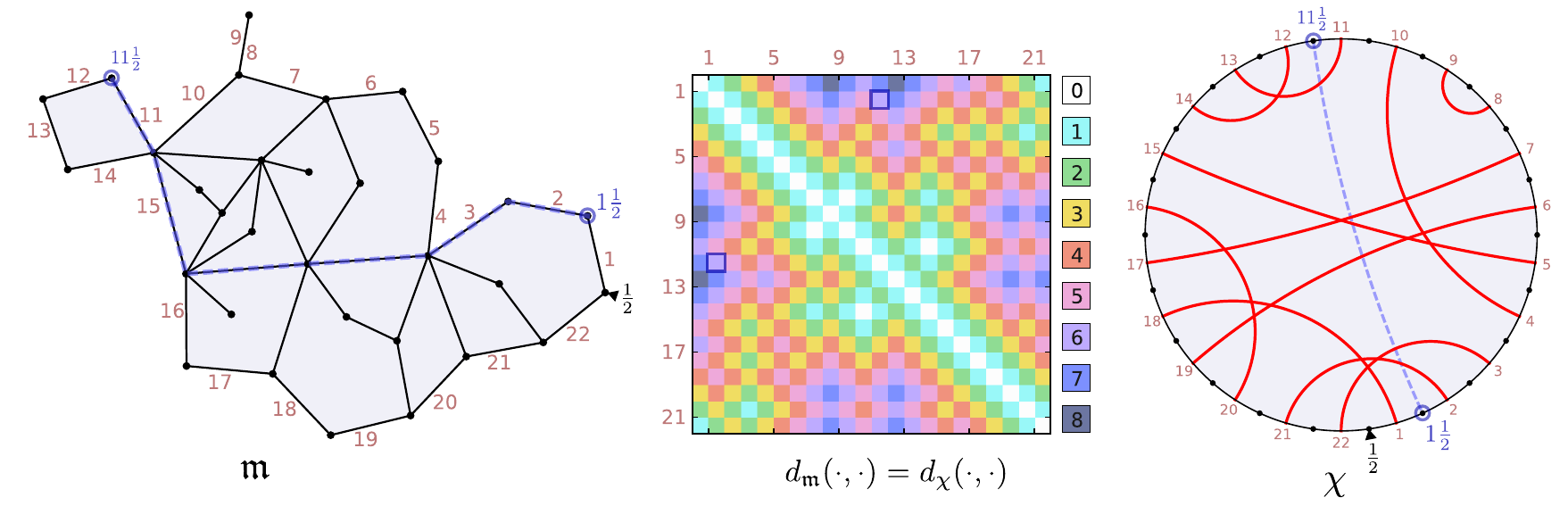}
	\caption{A bipartite planar map $\map$ (left) and a chord diagram $\chi$ (right) sharing the same pseudometric on $\{ \tfrac12, 1\tfrac12,\ldots,21\tfrac12\}$ (middle). A geodesic of length $d_{\map}(1\tfrac12,11\tfrac12)=6$ in $\map$ and a slice in $\chi$ with $d_{\chi}(1\tfrac12,11\tfrac12) = 6$ crossing chords are shown in blue.\label{fig:mapandchord}} 
\end{figure}

\section{Introduction}

The Sachdev-Ye-Kitaev (\emph{SYK}) model \cite{Sachdev1993,Sachdev2010,Kitaev2015} is a quantum-mechanical system of interacting Majorana fermions that has received a lot of attention in recent years in light of proposed holographic dualities with two-dimensional quantum gravity models \cite{Maldacena2016,Jensen2016,Cotler2017,Kitaev2018}.
Given the strong link between two-dimensional quantum gravity and discrete random geometry via random planar maps (see \cite{Gwynne2020,Sheffield2023,Budd2023} for recent reviews), it is natural to explore random geometric aspects to the holographic dictionary.

A natural starting point is the double-scaled limit of SYK (known as \emph{DSSYK}) which admits a combinatorial formulation in terms of chord diagrams  \cite{Berkooz2018,Berkooz2018a,Berkooz2025} with a free parameter $q \geq 0$, which has the interpretation as a Boltzmann weight for each pairwise crossing of chords.
This chord diagram system already has a random geometric flavor to it, which becomes even more apparent in the $q\to 1$ limit.
It is predicted \cite{Maldacena2016,Kitaev2018,Lin2022} that in this limit one recovers the hyperbolic geometry of Jackiw-Teitelboim (\emph{JT}) gravity \cite{Jackiw1985,Teitelboim1983,Mertens2023} on the Poincar\'e disk and its Schwarzian field theory description on the circle.
It has further been proposed that DSSYK at $q \neq 1$ may be dual to a $q$-deformed version of JT gravity \cite{Berkooz2023,Lin2022,Lin2023,Blommaert2024,Blommaert2025,Bossi2025} where geodesic distances are quantized.

An opportunity for a two-dimensional geometry interpretation of DSSYK is provided by the observation in \cite[Section~VIII.D]{Jafferis2023} that the partition function of DSSYK agrees with the large-$N$ limit of a special Hermitian 1-matrix model.
It is often referred to as the \emph{ETH 1-matrix model} after the Eigenstate Thermalization Hypothesis (ETH), which is an ansatz for the quantum-chaotic properties of operators featuring in the SYK model. 
Since the matrix model admits a formal power series expansion involving the enumeration of combinatorial planar maps (also known as ribbon graphs or 't~Hooft diagrams) with special degree-dependent weights associated to the faces, it leads to an identification of the DSSYK partition function with that of a special planar map model, that we refer to as the \emph{ETH planar map model}.

The main purpose of this work is to demonstrate that the correspondence between DSSYK and the ETH planar map model extends significantly beyond the partition function by relating chord number statistics to graph distances between boundary vertices.

Graph distance statistics in planar map models with general face weights have been studied extensively in past decades via a variety of combinatorial methods, including peeling processes \cite{Watabiki1995,Ambjoern1995,Angel2003,Budd2016,Curien2023}, tree bijections \cite{Albenque2015,Bouttier2004,Schaeffer1998} and slice decompositions \cite{Bouttier2012,Bouttier2019}.
Particularly the tree bijections were crucial in obtaining metric scaling limits of planar map models towards the Brownian sphere \cite{LeGall2013,Miermont2013} (and its stable cousins \cite{LeGall2010,Curien2025}), which was later demonstrated to agree with the metric of Liouville gravity \cite{Gwynne2021,Miller2020}.
When it comes to exact graph distance statistics at the discrete level, geodesic 2-point functions \cite{Ambjoern1995,Bouttier2003,Bouttier2012,Ambjoern2013,Ambjoern2016}, which control the distance between a pair of uniform vertices, and 3-point functions \cite{Bouttier2008,Fusy2014,Ambjoern2016}, controlling jointly the distances between a triple of uniform vertices, are known for various planar maps models on the 2-sphere.
For the disk topology, one can use tree bijections or slice decompositions to control distances from all boundary vertices to a single uniform bulk vertex \cite{Bouttier2009,Bettinelli2015}, or to a single distinguished boundary vertex \cite{Bouttier2012}.
What these exact results have in common is that the controlled graph distances can be realized by geodesics in the planar map that are pairwise non-crossing.
This is generally not the case, for instance, for geodesic $n$-point functions on the sphere for $n \geq 4$ (see also the discussion in \cite[Sec.~1.5]{Bouttier2019}).
We will show that for the ETH planar map model it is possible to control distances between all pairs of boundary vertices simultaneously.
As far as we are aware this is the first example of such enumerative control involving (many) crossing geodesics.
We note, however, that the result is non-probabilistic a priori, since the ETH planar map model involves weights of alternating sign.

\subsection{Main result}

The two combinatorial families featuring in this work are chord diagrams and bipartite planar maps, so we start by introducing these families. 
For real numbers $a \leq b$ such that $b-a$ is an integer we denote by $\interval{a,b}\coloneqq\{a,a+1,\ldots,b\}$ the (shifted) integers in the interval $[a,b]$.
A \emph{chord diagram} $\chi$ of size $|\chi| = 2n$ is a fixed-point-free involution on $\interval{1,2n}$.
Equivalently, it is a partition of $\interval{1,2n}$ into two-element sets, called the \emph{chords} of $\chi$, which are conveniently illustrated by positioning $\interval{1,2n}$ in counterclockwise order on the circle with a chord between $i$ and $\chi(i)$ (see the right illustration in Figure~\ref{fig:mapandchord}). 
We denote by $\chords$ the set of chord diagrams and by $\chords_{2n} \subset \chords$ those of size $2n$.
In general, for $a<a'$ and $b<b'$, we say $\{a,a'\}$ \emph{crosses} $\{b,b'\}$ whenever $a < b < a' < b'$ or $b < a < b' < a'$.
The total number of crossings of $\chi$ is
\begin{align}
	\intersec(\chi) = |\{ \text{pairs }\{A,B\}\text{ of chords of }\chi : A\text{ crosses }B \}|.
\end{align}
For half-integers $a,b \in \interval{\tfrac12,2n-\tfrac12}$, we denote by $d_\chi(a,b)$ the number of chords crossing $\{a,b\}$,
\begin{align}
	d_\chi(a,b) = |\{ \text{chords }B\text{ of }\chi : B\text{ crosses }\{a,b\}\}|.\label{eq:chorddist}
\end{align} 
This is often referred to in the physics literature as the \emph{chord number} for the slice between $a$ and $b$ \cite{Lin2022,Berkooz2023}.
A chord crossing $\{a,b\}$ must also cross at least one of $\{a,c\}$ or $\{c,b\}$ for every $c\in \interval{\tfrac12,2n-\tfrac12}$, implying the triangle inequality $d_{\chi}(a,b) \leq d_{\chi}(a,c) + d_\chi(c,b)$.
Therefore $d_{\chi}$ turns $\interval{\tfrac12,2n-\tfrac12}$ into a pseudometric space, because in addition it is symmetric and obeys $d_{\chi}(a,a) = 0$.
It is generally not a metric because we may have $d_\chi(a,b)=0$ for $a\neq b$ when no chords separate $a$ and $b$. 
Observe that $\chi$ is uniquely determined by its pseudometric.

A \emph{(planar) map} $\map$ is a connected (multi)graph embedded properly in the plane, viewed up to orientation-preserving homeomorphisms of the plane (see the left illustration in Figure~\ref{fig:mapandchord}).
We always assume a planar map to carry a distinguished corner of the outer face, called the \emph{root corner} (indicated by the little triangle in Figure~\ref{fig:mapandchord}).
The \emph{degree} of a face of $\map$ is the number of corners in that face, and the \emph{perimeter} $|\partial\map|$ is the degree of the outer face.
A map that has all faces of even degree is called \emph{bipartite}.
We denote by $\mathcal{M}$ the set of bipartite planar maps and by $\mathcal{M}_{2n}$ those of perimeter $2n$.
We may naturally label the corners in the outer face of a bipartite map $\map$ of perimeter $2n$ by $\interval{\tfrac12,2n-\tfrac12}$ in counterclockwise order, assigning $\tfrac12$ to the root corner.
Then we obtain a natural pseudometric $d_\map:\interval{\tfrac12,2n-\tfrac12}^2 \to \Z_{\geq0}$ by setting $d_\map(a,b)$ to be the graph distance in $\map$ between the pair of (vertices adjacent to the) corners with labels $a$ and $b$.
As we will see in Proposition~\ref{prop:mapchord} below, there exists a unique chord diagram that we denote by $\chi = \mathsf{Geod}(\map)$ and call the \emph{geodesic chord diagram} of $\map$, such that $d_\map = d_\chi$.
We denote by $\maps_\chi \subset \maps$ the collection of maps that have $\chi$ as their geodesic chord diagram.

Given a weight sequence $\mathbf{t} = (t_2,t_4,\ldots)$, it is customary to consider the \emph{Boltzmann weight}
\begin{align}
	w_{\mathbf{t}}(\map) = \prod_{f\in \mathsf{Faces}(\map)} t_{\deg(f)}
\end{align}
for a map $\map \in \maps$, where the product runs over all faces $\mathsf{Faces}(\map)$ excluding the outer face.
This corresponds precisely to the weight associated to such a map in the expansion of a matrix model with \emph{potential derivative}
\begin{equation}
	V'(x) = x - \sum_{k\geq 1} t_{2k} x^{2k-1}.
\end{equation}
The corresponding \emph{disk function} for maps of perimeter $2n$ is the formal multivariate generating series
\begin{equation}
	F_{2n}(\mathbf{t}) = \sum_{\map\in\maps_{2n}} w_{\mathbf{t}}(\map).
\end{equation}
Since these maps naturally partition according to their geodesic chord diagram, we introduce the \emph{geodesic disk function}
\begin{align}
	F(\mathbf{t};\chi) = \sum_{\map\in \maps_\chi} w_{\mathbf{t}}(\map),\label{eq:geoddiskfunction}
\end{align}
so that $F_{2n}(\mathbf{t}) = \sum_{\chi \in \chords_{2n}} F(\mathbf{t};\chi)$.
By convention, we let $F(\mathbf{t};\emptyset) = 1$ for the empty chord diagram (of size $0$), resulting from the bipartite planar map consisting of a single vertex (of perimeter $0$).
It is a very difficult problem to characterize these series for all but the simplest chord diagrams $\chi$, because it involves controlling distances realized by geodesics in $\map$ that may intersect in many ways.

However, the situation is different for the ETH planar map model with parameter $q$, which corresponds to the choice of potential derivative \cite[Sec.~VIII.D]{Jafferis2023} 
\begin{align}
	V_q'(x) = 2\sqrt{1-q} \sum_{\ell\geq 1} (-1)^{\ell-1} q^{\binom{\ell}{2}} T_{2\ell-1}\left(\frac{\sqrt{1-q}}{2} x\right),
\end{align}
where $T_{n}$ is the Chebyshev polynomial of the first kind satisfying $T_n(\cos \theta) = \cos(n\theta)$.
The corresponding weight sequence $\mathbf{t}(q) = (t_2(q),t_4(q),\ldots)$ is
\begin{align}
	t_{2k}(q) = \delta_{k,1} + (q-1)^k \sum_{\ell\geq k} q^{\binom{\ell}{2}} \frac{2\ell-1}{2k-1} \binom{\ell+k-2}{\ell-k},\label{eq:specweights}
\end{align}
which we understand as formal power series in $q$.
These weights were fine-tuned in \cite{Jafferis2023} by requiring their disk function to match the DSSYK partition function,
\begin{align}
	F_{2n}(\mathbf{t}(q)) = \sum_{\chi\in\chords_{2n}} q^{\intersec(\chi)} = m_{2n}(q),\label{eq:partitionfunction}
\end{align}
where $m_{2n}(q)$ is the Touchard-Riordan polynomial \cite{Touchard1952,Riordan1975} (see Section~\ref{sec:enumintro}).
Our main result is that in this case the geodesic disk function only depends on the crossing number $\intersec(\chi)$. 

\begin{theorem}\label{thm:main}
	For every $\chi \in \chords$ the geodesic disk function \eqref{eq:geoddiskfunction} obeys the formal power series identity
\begin{equation}
	F(\mathbf{t}(q);\chi) = q^{\intersec(\chi)}.\label{eq:mainidentity}
\end{equation}
\end{theorem}

\noindent
Since $F_{2n}(\mathbf{t}(q)) = \sum_{\chi \in\chords_{2n}} F(\mathbf{t}(q);\chi)$, this gives a new proof of the equality of partition functions \eqref{eq:partitionfunction} while presenting a significant refinement of the correspondence.

\paragraph{Correlation functions.}
One way of appreciating the refined correspondence is by noting that all joint moments of the pseudometric observables agree between the two models.
Indeed, if $n,k\geq 1$ and $a_1,\ldots,a_k,b_1,\ldots,b_k \in \interval{\tfrac12,2n-\tfrac12}$ are arbitrary points on the circle and $f: \mathbb{Z}_{\geq 0}^k \to \mathbb{R}$ an arbitrary function, then Theorem~\ref{thm:main} shows that 
\begin{align}
	\sum_{\map\in\maps_{2n}} f\big(d_{\map}(a_1,b_1),\ldots,d_{\map}(a_k,b_k)\big)\,w_{\mathbf{t}(q)}(\map) = \sum_{\chi\in\chords_{2n}} f\left(d_{\chi}(a_1,b_1),\ldots,d_{\chi}(a_k,b_k)\right) q^{\intersec(\chi)}.\label{eq:correlations}
\end{align}
Correlation functions in DSSYK are more often discussed in terms of chord diagrams \emph{coupled to matter} \cite{Berkooz2018a,Berkooz2018,Lin2023}, so let us rephrase our results in this language.
A \emph{two-type chord diagram} is a chord diagram in which each chord is colored red or blue, and called, respectively, a \emph{Hamiltonian chord} and a \emph{matter chord} (see Figure~\ref{fig:correlationfunction}).
To a two-type chord diagram $\tau$ with $k$ Hamiltonian chords and $n$ matter chords we may naturally associate a Hamiltonian chord diagram $\chi \in \chords_{2k}$ and a matter chord diagram $\mu \in \chords_{2n}$ by deleting all chords of the other type.
We introduce the positions $s_1, \ldots, s_{2n}\in\interval{\tfrac12,2k-\tfrac12}$ by letting $s_j - \tfrac12$ be the number of endpoints of Hamiltonian chords in $\tau$ preceding the $j$th endpoint of a matter chord.
Then the \emph{matter correlator} can be defined as
\begin{align}
	C_{2k, \mu}^{s_1,\ldots,s_{2n}}(q,\tilde{q}) \coloneqq\sum_{\tau} q^{\intersec_{\mathrm{HH}}(\tau)}\tilde{q}^{\intersec_{\mathrm{MH}}(\tau)}, \label{eq:mattercorr}
\end{align}
where the sum is over all two-type chord diagrams $\tau$ of size $2k+2n$ sharing the same matter chord diagram $\mu$ and positions $s_1, \ldots, s_{2n}$, while $\intersec_{\mathrm{HH}}(\tau)$ counts the crossings between pairs of Hamiltonian chords and $\intersec_{\mathrm{MH}}(\tau)$ the crossings between matter chords and Hamiltonian chords.
For example, the case $\mu=\{\{1,2\}\}$ gives the 2-point function, $\mu = \{\{1,2\},\{3,4\}\}$ the uncrossed 4-point function and $\mu = \{\{1,3\},\{2,4\}\}$ (as in Figure~\ref{fig:correlationfunction}) the crossed 4-point function or \emph{out-of-time order correlator} \cite{Berkooz2018,Lin2023,Jafferis2023}.
\begin{figure}[h]
	\centering
	\includegraphics[width=.94\linewidth]{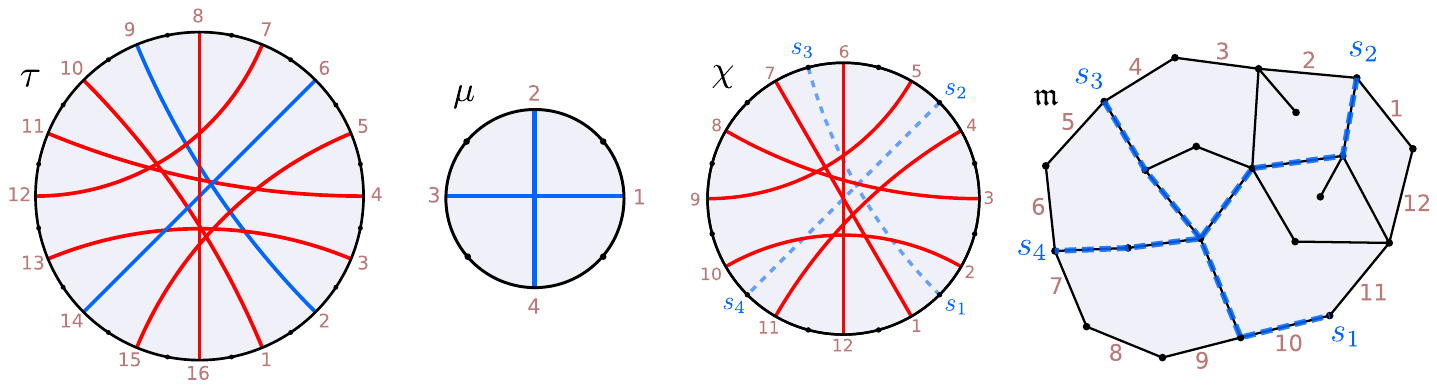}
	\caption{Example of a two-type chord diagram $\tau$ with $k=6$ Hamiltonian chords (red) and $n=2$ matter chords (blue). In this case there are $\intersec_{\mathrm{MH}}(\tau) = 9$ crossings between matter and Hamiltonian chords. The diagram $\tau$ is equivalently described by the matter chord diagram $\mu$ and the Hamiltonian chord diagram $\chi$, together with the positions $s_1,\ldots,s_{2n}$ of the endpoints of matter chords in the latter ($1\tfrac12,4\tfrac12,6\tfrac12,10\tfrac12$ in this case). If $\chi = \mathsf{Geod}(\map)$ is the geodesic chord diagram of a planar map $\map$, then $\intersec_{\mathrm{MH}}(\tau)$ equals the sum over all chords $\{a,b\}$ of $\mu$ of the graph distance between the corners labelled $s_a$ and $s_b$. Geodesics realizing these graph distances with total length $\intersec_{\mathrm{MH}}(\tau) = 9$ are shown in dashed blue. \label{fig:correlationfunction}}
\end{figure}

The two-type chord diagrams $\tau$ appearing in \eqref{eq:mattercorr} are bijectively encoded by their Hamiltonian chord diagram $\chi$ such that $\intersec_{\mathrm{HH}}(\tau) = \intersec(\chi)$ and the number of Hamiltonian chords crossed by the chord corresponding to the chord $\{j,\mu(j)\}$ of $\mu$ is $d_{\chi}(s_j,s_{\mu(j)})$. 
The matter correlator is thus of the form of the right-hand side of \eqref{eq:correlations}, so that we can identify it with the planar map correlator
\begin{align}
	C_{2k, \mu}^{s_1,\ldots,s_{2n}}(q,\tilde{q}) = \sum_{\map\in\maps_{2k}}  \tilde{q}^{\frac12 \sum_{j=1}^{2n} d_{\map}(s_j,s_{\mu(j)})}w_{\mathbf{t}(q)}(\map),
\end{align}
where the $\tfrac12$ in the exponent appears because the sum counts each chord twice.

\paragraph{Formal versus real series.} 
Theorem~\ref{thm:main} is formulated at the level of formal power series in $q$, which is justified by the fact that at each order in $q$ only a finite number of maps contribute non-trivially.
One may wonder what happens if we take $q$ to be a positive real number.
The series appearing in the weights $t_{2k}(q)$ in \eqref{eq:specweights} each have radius of convergence equal to $1$, but for any $0< q < 1$ the sign of $t_{2k}(q)$ is alternating in $k$.
As a consequence, the sign of $w_{\mathbf{t}(q)}(\map)$ is $(-1)^{\mathsf{V}(\map)-n-1}$ where $\mathsf{V}(\map)$ is the number of vertices of $\map$ and $2n$ its perimeter.
This means that the sum in the geodesic disk function \eqref{eq:geoddiskfunction} converges absolutely only when the geodesic disk function $F(|\mathbf{t}(q)|;\chi)$ with positive weights $|w_{\mathbf{t}(q)}(\map)| = w_{|\mathbf{t}(q)|}(\map)$ converges.
It is not difficult to see that the convergence does not depend on the choice of $\chi$ (as long as $\chi \neq \emptyset$) and is equivalent to the disk function $F_{2n}(|\mathbf{t}(q)|) < \infty$ converging for some $n \geq 1$. 
In the literature on planar maps a weight sequence $|\mathbf{t}(q)|$ satisfying the latter condition is said to be \emph{admissible}, see \cite{Marckert2007} and \cite[Sec.~3.3.2]{Curien2023}.
This is known \cite{Marckert2007} to be equivalent to the equation
\begin{align}
	r - \sum_{k=1}^\infty |t_{2k}(q)|\, \binom{2k-1}{k} r^k = 1\label{eq:admissible}
\end{align}
having a positive solution $r > 0$.
As will be checked in Corollary~\ref{cor:probabilistic} below, this is the case for sufficiently small $q$ but not for $q$ close to $1$.
Of course, if it converges absolutely, then $F(\mathbf{t}(q);\chi)=q^{\intersec(\chi)}$ holds as equality in the reals.

\paragraph{Probabilistic interpretation.} Even though the alternating sign of $t_{2k}(q)$ prevents a direct probabilistic interpretation of Theorem~\ref{thm:main}, the previous discussion provides a somewhat artificial one.
If $\mathbf{t}=(t_2\geq 0,t_4\geq 0, \ldots)$ is an admissible weight sequence and $F_{2n}(\mathbf{t})$ is the corresponding disk function, then the \emph{$\mathbf{t}$-Boltzmann planar map} $M_{2n}$ of perimeter $2n$ is the random map with probability distribution $\mathbb{P}(M_{2n} = \map) = w_{\mathbf{t}}(\map) / F_{2n}(\mathbf{t})$ for $\map\in\maps_{2n}$.

\begin{corollary}\label{cor:probabilistic}
	The weight sequence $|\mathbf{t}(q)|$ is admissible for $0 < q < 0.042$ and not admissible for $0.21 < q \leq 1$. 
	If $|\mathbf{t}(q)|$ is admissible, then the $|\mathbf{t}(q)|$-Boltzmann planar map $M_{2n}$ satisfies
	\begin{equation}
		\mathbb{P}(\mathsf{Geod}(M_{2n}) = \chi, \,\mathsf{V}(M_{2n})\text{ even}) \,\,-\,\, \mathbb{P}(\mathsf{Geod}(M_{2n}) = \chi,\,\mathsf{V}(M_{2n})\text{ odd})= \frac{(-1)^{n-1}}{F_{2n}(|\mathbf{t}(q)|)}\, q^{\intersec(\chi)}
	\end{equation}
	for all $n \geq 1$ and $\chi \in \mathcal{C}_{2n}$. 
\end{corollary}

Based on numerical evaluation, it appears that $|t_{2k}(q)|$ is monotonically increasing in $q \in [0,1]$ for each $k\geq 1$, but we were unable to prove this.
If true it would follow that there exists a unique critical value $q_* \in [0.042,0.21]$, numerically given by $q_*\approx 0.0694$, such that $|\mathbf{t}(q)|$ is admissible if and only if $q \leq q_*$.
Moreover, the $|\mathbf{t}(q)|$-Boltzmann planar map $M_{2n}$ would be \emph{critical}, in the probabilistic sense that the number $\mathsf{V}(M_{2n})$ of vertices of $M_{2n}$ has infinite variance (see \cite{Marckert2007} and \cite[Sec.~5.2]{Curien2023}), precisely when $q=q_*$.

\subsection{Questions}
This work raises several natural follow-up questions:
\begin{enumerate}[1.]
	\item The identity \eqref{eq:mainidentity} of Theorem~\ref{thm:main} relies on massive cancellations, since $F(\mathbf{t}(q);\chi) = \sum_{\map\in\maps_\chi}w_{\mathbf{t}(q)}(\map)$ involves infinitely many contributions of maps $\map$ whose weight is $w_{\mathbf{t}(q)}(\map) = O(q^{\intersec(\chi)+1})$. Our proof observes these cancellations only indirectly in the final stage (Proposition~\ref{prop:tuttechord}), see also the outline in Section~\ref{sec:outline} below. \emph{Is there a simpler proof of Theorem~\ref{thm:main} by grouping maps in $\maps_\chi$ in a way that makes the cancellation explicit?}
	\item The limit $q\to 1$ is of significant interest because of the connection between SYK and the Schwarzian theory and JT gravity \cite{Maldacena2016,Kitaev2018,Saad2019}.
	On the map enumeration side, it was proposed by Okuyama in \cite{Okuyama2023} and later verified by Giacchetto, Maity \& Mazenc \cite{Giacchetto2025} and Do \& Norbury \cite{Do2025a}, that certain polynomials appearing in the topological recursion for the weights \eqref{eq:specweights} approach as $q\to 1$ the Weil-Petersson volumes of hyperbolic surfaces with geodesic boundaries.
	Unfortunately, this limit is far outside the regime of absolute convergence, so one has to take care in interpreting the planar map correspondence there.
	\emph{Does there exist a reasonable subfamily $\mathcal{M}' \subset \mathcal{M}$ of bipartite planar maps together with a natural weight function $w'_q(\map)$ such that $\sum_{\map \in \maps'_\chi} w_q'(\map)$ still gives $q^{\intersec(\chi)}$ with a regime of absolute convergence that extends towards $q=1$?}
	\item \emph{Is there an analogous relation between boundary metrics of maps and chord-like diagrams for topologies other than the disk, i.e.\ for surfaces of higher genus or with multiple boundaries?} The topological recursion of maps with weights \eqref{eq:specweights} has been studied in \cite{Okuyama2023}, giving rise to certain $q$-deformed Weil-Petersson volumes \cite{Okuyama2023,Giacchetto2025,Do2025a}. It is natural to address this question in light of a recent bijective approach to maps with so-called tight boundaries \cite{Bouttier2022,Bouttier2024}.
	\item The concrete correspondence of Theorem~\ref{thm:main} between the metric of planar maps and chord diagrams suggests that the planar map model with weights \eqref{eq:specweights}, arising in the large-$N$ limit of the ETH 1-matrix model, has something precise to say about a holographic gravitational dual of DSSYK. In particular, it elucidates the discreteness of bulk geodesic lengths by their identification with graph distances in the planar map. Several proposals for a bulk dual have been made in the literature, including a noncommutative $q$-deformation of the hyperbolic plane (or AdS$_2$) \cite{Berkooz2023,Lin2022,Lin2023}, and sine-dilaton gravity \cite{Blommaert2024,Blommaert2025,Bossi2025}. \emph{Is there a relation between the bulk Hilbert spaces in these proposals and the natural bulk Hilbert space one obtains from the map model by slicing along geodesics?}
\end{enumerate}

\subsection{Idea of the proof and outline}\label{sec:outline}

The disk functions $F_{2n}(\mathbf{t}) = \sum_{\map\in\maps_{2n}} w_{\mathbf{t}}(\map)$ of bipartite planar maps satisfy the Tutte equations \cite{Tutte1962}, known better as loop equations or Schwinger-Dyson equations in the physics literature,
\begin{align}
	F_{2n}(\mathbf{t}) = \sum_{p\geq 1} t_{2p} F_{2n+2p-2}(\mathbf{t}) + \sum_{\ell=0}^{n-1} F_{2\ell}(\mathbf{t})F_{2n-2\ell-2}(\mathbf{t}).
\end{align}
This follows directly from considering what happens when one deletes from a map $\map$ the edge sitting directly to the left of the root corner on the outer face: the first term accounts for the possibility of revealing a face of degree $2p$, leaving a map $\map_0$ with perimeter $2n+2p-2$, and the second for the situation where deletion disconnects the map into a pair of maps $\map_1$, $\map_2$ with perimeter $2\ell$ and $2n-2\ell-2$ respectively.
The general idea of the proof is to show that the geodesic disk functions $F(\mathbf{t};\chi)$ satisfy suitable refinements of the Tutte equations and that they have the unique solution $q^{\intersec(\chi)}$ in the case $\mathbf{t}=\mathbf{t}(q)$.

The reason that such refinements exist is that the boundary metric $d_\map$ of $\map$, and therefore also its geodesic chord diagram $\mathsf{Geod}(\map)$, is uniquely determined by the boundary metric $d_{\map_0}$ of the map $\map_0$ or the boundary metrics $d_{\map_1}, d_{\map_2}$ of the maps $\map_1$ and $\map_2$ obtained after the edge deletion.
The relation in the latter case is easy: the chord diagram $\mathsf{Geod}(\map)$ is obtained via a simple concatenation of $\mathsf{Geod}(\map_1)$ and $\mathsf{Geod}(\map_2)$.

The first case is more complicated: $\map$ is obtained from $\map_0$ by drawing an edge in the outer face of $\map_0$, which introduces a shortcut for the boundary metric and can therefore change the geodesic chord diagram significantly.
In Section~\ref{sec:shortcut} we identify this shortcutting operation on the level of chord diagrams by introducing a mapping $\mathsf{Short}_{n,p}$ such that $\mathsf{Geod}(\map) = \mathsf{Short}_{n,p}(\mathsf{Geod}(\map_0))$.
Subsequently, we classify (in Proposition~\ref{prop:shortsufficient}) all chord diagrams $\chi'$ such that $\mathsf{Short}_{n,p}(\chi') = \chi$.

In Section~\ref{sec:enum} we discuss the enumeration of chord diagrams with weight $q$ per crossing.
In particular, for any fixed chord diagram $\chi$ we explicitly enumerate the chord diagrams $\chi'$ satisfying $\mathsf{Short}_{n,p}(\chi') = \chi$ via a bijective decomposition of $\chi'$.
Proposition~\ref{prop:tuttechord} then verifies the refined Tutte equations at the level of chord diagrams.
It relies on a ``miracle'' for the weight sequence \eqref{eq:specweights} by which the final expression in the Tutte equation becomes a summation of the $q$-Stirling numbers of the first kind appearing in Lemma~\ref{lem:potentialinqhermite} against the $q$-Stirling numbers of the second kind from Proposition~\ref{prop:specialslice}, which are orthogonal and thus trivialize the summation.
In Section~\ref{sec:mainproof} we prove Theorem~\ref{thm:main} by showing the uniqueness of the solution via an inductive argument.
Corollary~\ref{cor:probabilistic} then follows with help of a few bounds for $0 < q < 1$.

\subsection*{Acknowledgments}

This work is part of the VIDI programme with project number VI.Vidi.193.048, which is financed by the Dutch Research Council (NWO).
The author gratefully acknowledges support from the Simons Center for Geometry and Physics, Stony Brook University, at which some of the research for this paper was performed during the program \emph{Random Geometry in Math and Physics}, 23 March - 1 May 2026.

\section{Chord diagrams and planar maps}

\subsection{Geodesic chord diagrams associated to planar maps}
Recall that we denote the chord diagrams of size $2n$ by $\mathcal{C}_{2n}$ and the bipartite planar maps of perimeter $2n$ by $\maps_{2n}$.
For $\chi\in\mathcal{C}_{2n}$ and $a\in \interval{1,2n}$, we will further use the notation 
\begin{align}
	\ell_\chi(a) = |\{ \text{chord }B\text{ of }\chi : B\text{ crosses }\{a,\chi(a)\}\}|
\end{align}
for the number of chords crossing the chord $\{a, \chi(a)\}$.
Note that we then have the relations
\begin{align}
	d_\chi(a \pm \tfrac12, \chi(a) \mp \tfrac12) = \ell_\chi(a), \quad d_\chi(a \pm \tfrac12, \chi(a) \pm \tfrac12) = \ell_\chi(a) +1, \label{eq:parallelchord}
\end{align}
where we implicitly identify $2n + \tfrac12$ and $\tfrac12$.

\begin{proposition}\label{prop:mapchord}
	For each map $\map \in \maps$ there exists a unique chord diagram $\mathsf{Geod}(\map) \in \chords$ such that $d_\map = d_{\mathsf{Geod}(\map)}$. Moreover, $\mathsf{Geod} : \maps \to \chords$ is surjective. 
\end{proposition}
\begin{proof}
	Let $n\geq 1$ and $\map \in \maps_{2n}$. 
	Denote by $e_1, \ldots, e_{2n}$ the edges on the outer face of $\map$ in counterclockwise order starting from the root corner (like the labeling in Figure~\ref{fig:mapandchord}).
	We define $\mathsf{Geod}(\map)$ explicitly as an involution on $\interval{1,2n}$ as follows.
	For $a\in\interval{1,2n}$, let us color each vertex of $\map$ depending on which endpoint of $e_a$ is closest to it in the graph distance metric, observing that no ties arise because $\map$ is bipartite.
	Concretely, we color a vertex white if it is closer to the corner $a-\tfrac12$ and black if it is closer to the corner $a+\tfrac12$ (see Figure~\ref{fig:voronoi}).
	Since a shortest path between a white vertex and the corner $a-\tfrac12$ necessarily encounters only white vertices, the white cluster is connected.
	The same is true for the black cluster.
	It follows from planarity of $\map$ that the interface between the white and black cluster, starting at $e_a$, ends at a unique bicolored edge $e_b$ on the outer face with $b\neq a$.
	We then set $\mathsf{Geod}(\map)(a) = b$.

	\begin{figure}[h]
		\centering
		\includegraphics[width=.5\linewidth]{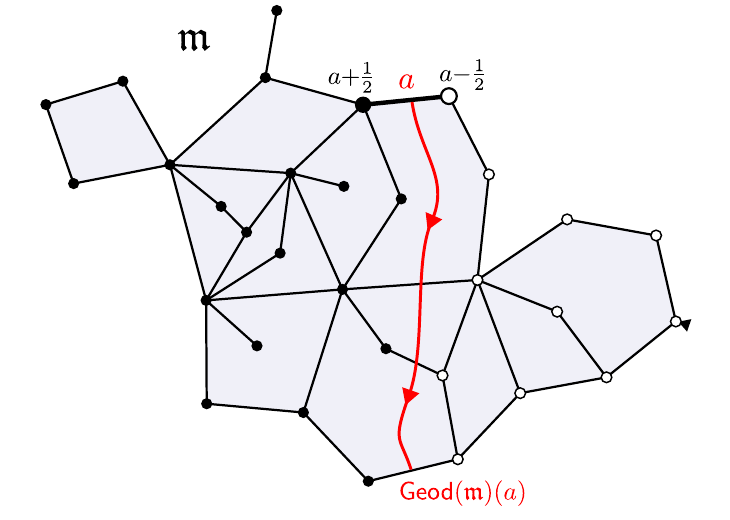}
		\caption{The Voronoi diagram of the corners $a \pm \tfrac12$ adjacent to $a$ singles out a unique ``parallel'' edge labelled $b=\mathsf{Geod}(\map)(a)$.}\label{fig:voronoi}
	\end{figure}

	This defines a fixed-point-free mapping $\mathsf{Geod}(\map) : \interval{1,2n} \to \interval{1,2n}$.
	To see that it is an involution, note that $\mathsf{Geod}(\map)(a) = b$ by construction is equivalent to
	\begin{align}
		d_\map(a-\tfrac12,b+\tfrac12) < d_\map(a+\tfrac12,b+\tfrac12)\quad\text{and}\quad d_\map(a+\tfrac12,b-\tfrac12) < d_\map(a-\tfrac12,b-\tfrac12).\label{eq:geodcondition}
	\end{align}
	By triangle inequalities this in turn is equivalent to 
	\begin{align*}
		d_\map(a-\tfrac12,b+\tfrac12) = d_\map(a+\tfrac12,b-\tfrac12) = d_\map(a+\tfrac12,b+\tfrac12)-1 = d_\map(a-\tfrac12,b-\tfrac12)-1.
	\end{align*}
	This condition is symmetric in $a$ and $b$, so is also equivalent to $\mathsf{Geod}(\map)(b) = a$.
	We conclude that $\mathsf{Geod}(\map) \in \mathcal{C}_{2n}$.

	Let us now check that $d_{\chi} = d_{\map}$ when $\chi = \mathsf{Geod}(\map)$. 
	Since both are pseudometrics, it suffices to check that their increments agree, i.e.\ that for all $i\in \interval{\tfrac12,2n-\tfrac12}$ and $a\in \interval{1,2n}$ we have
	\begin{align*}
		d_{\chi}(i,a+\tfrac12) - d_{\chi}(i,a-\tfrac12) = d_{\map}(i,a+\tfrac12) - d_{\map}(i,a-\tfrac12).
	\end{align*}
	The right-hand side is $1$, respectively $-1$, if corner $i$ is closer to corner $a-\tfrac12$, respectively $a+\tfrac12$.
	In the first case the chord of $\mathsf{Geod}(\map)$ starting at $a$ crosses $\{i,a+\tfrac12\}$ and does not cross $\{i,a-\tfrac12\}$, while in the second case it is the other way around.
	This matches precisely the left-hand side, showing that $d_{\chi} = d_{\map}$.
	Since a chord diagram is uniquely characterized by its pseudometric $d_{\chi}$, this establishes the claimed existence and uniqueness of the geodesic chord diagram.

	Finally, let us show that for each $\chi \in \mathcal{C}_{2n}$ there exists a map $\map\in\maps_{2n}$ such that $\mathsf{Geod}(\map) = \chi$.
	For this let us fix any drawing $\tilde{\chi}$ of the chord diagram $\chi$ in the unit disk such that pairs of chords intersect at most once and do so transversally, e.g.\ as in Figure~\ref{fig:mapandchord}.
	We interpret $\tilde{\chi}$ as a planar map by putting a vertex at each crossing and each endpoint of a chord and letting segments of chords or of the boundary circle be the edges of $\tilde{\chi}$ (Figure~\ref{fig:dualmap}).
	By construction $\tilde{\chi}$ has perimeter $2n$ and all inner vertices (i.e.\ the vertices corresponding to crossings) are of even degree.
	Then one can consider the \emph{dual map} $\tilde{\chi}^\dagger$, obtained by putting a new vertex in each inner face of $\tilde{\chi}$ and drawing for each inner edge $e$ of $\tilde{\chi}$ a dual edge of $\tilde{\chi}^\dagger$ connecting the vertices in the faces adjacent to $e$.
	Then $\tilde{\chi}^\dagger$ has perimeter $2n$ as well and all inner faces are of even degree.
	Hence $\tilde{\chi}^{\dagger} \in \maps_{2n}$.  

	\begin{figure}[h]
		\centering
		\includegraphics[width=.75\linewidth]{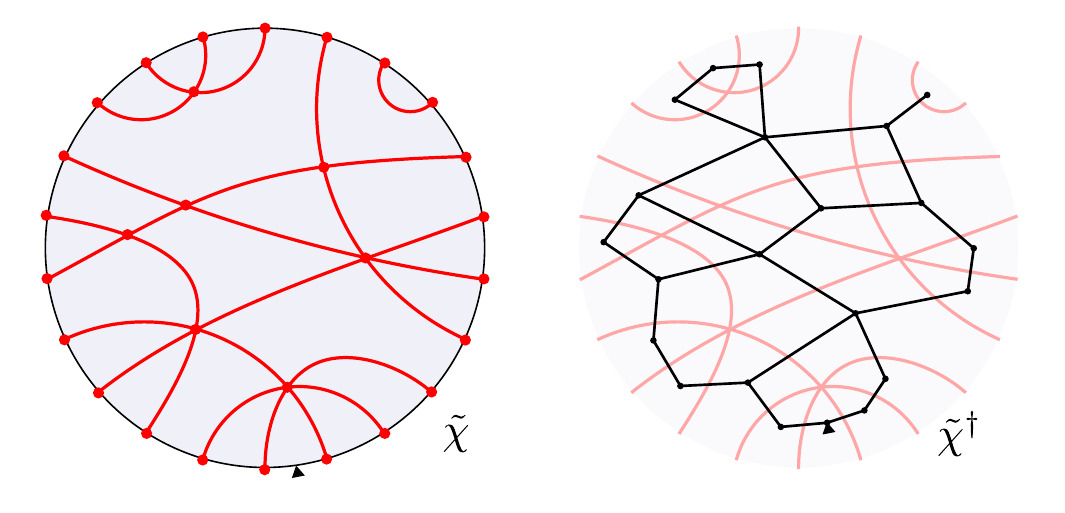}
		\caption{A drawing $\tilde{\chi}$ of the chord diagram $\chi\in\mathcal{C}$ of Figure~\ref{fig:mapandchord} together with its dual map $\tilde{\chi}^\dagger \in \maps$.}\label{fig:dualmap}
	\end{figure}

	It remains to show that for $\map = \tilde{\chi}^{\dagger}$ we have $\mathsf{Geod}(\map) = \chi$. 
	It is easy to see that for each $i,j\in\interval{\tfrac12,2n-\tfrac12}$ we have the inequality $d_{\map}(i,j) \geq d_\chi(i,j)$.
	Indeed, any vertex path in $\map$ from corner $i$ to corner $j$ corresponds to a face path in $\tilde{\chi}$ starting adjacent to $i$ and ending adjacent to $j$. 
	Each step of such a path crosses at most one chord and there are $d_\chi(i,j)$ chords to cross.
	On the other hand, for $a\in\interval{1,2n}$, consider the chord $\{a,b=\chi(a)\}$ which is crossed $\ell_\chi(a)$ times.
	We can find a path of length $\ell_\chi(a)$ in $\map$ from $a+\tfrac12$ to $b-\tfrac12$ by traveling along the faces of $\map$ that are crossed by the chord.
	Therefore
	\begin{align*}
		d_\map(a+\tfrac12,b-\tfrac12) \leq \ell_\chi(a) \stackrel{\eqref{eq:parallelchord}}{<} d_\chi(a+\tfrac12,b+\tfrac12) \leq d_\map(a-\tfrac12,b-\tfrac12).
	\end{align*}
	Analogously $d_\map(a-\tfrac12,b+\tfrac12) < d_\map(a+\tfrac12,b+\tfrac12)$.
	From \eqref{eq:geodcondition} we then conclude that $\mathsf{Geod}(\map)(a) = b = \chi(a)$.
	This proves the final statement.
\end{proof}

\subsection{Shortcutting chord diagrams}\label{sec:shortcut}

The main combinatorial construction at the level of chord diagrams that we need is that of introducing a \emph{shortcut}. For $n,p\geq 1$, let $\chi' \in \mathcal{C}_{2n+2p-2}$ and consider its pseudometric $d_{\chi'}$ on $\interval{\tfrac12,2n+2p-2-\tfrac12}$.
Then we can construct from it a pseudometric on $\interval{\tfrac12,2n-\tfrac12}$ by introducing a shortcut between $\tfrac12$ and $2n-\tfrac12$, meaning that we set 
\begin{align}
	d_{\chi}(a,b) = \min\Big[ d_{\chi'}(a,b), d_{\chi'}(a,\tfrac12) + d_{\chi'}(2n-\tfrac12,b) +1, d_{\chi'}(a,2n-\tfrac12) + d_{\chi'}(\tfrac12,b) +1 \Big]\label{eq:shortcutdef}
\end{align}
for $a,b\in\interval{\tfrac12,2n-\tfrac12}$.
Then this corresponds to a unique chord diagram $\chi \in \mathcal{C}_{2n}$ of size $2n$ that we denote by $\mathsf{Short}_{n,p}(\chi')=\chi$.
One way to see this is by observing that, according to Proposition~\ref{prop:mapchord}, there exists a bipartite planar map $\map'$ of perimeter $2n+2p-2$ so that $d_{\chi'} = d_{\map'}$.
From $\map'$ one can construct another bipartite planar map $\map$ of perimeter $2n$ by drawing an additional (``shortcut'') edge through the outer face of $\map'$ connecting $\tfrac12$ and $2n-\tfrac12$.
Then $d_{\map}$ is given exactly by the right-hand side of \eqref{eq:shortcutdef} and therefore corresponds to the pseudometric $d_\chi$ of $\chi = \mathsf{Geod}(\map) \in \mathcal{C}_{2n}$.

	\begin{figure}[h]
		\centering
		\includegraphics[width=.8\linewidth]{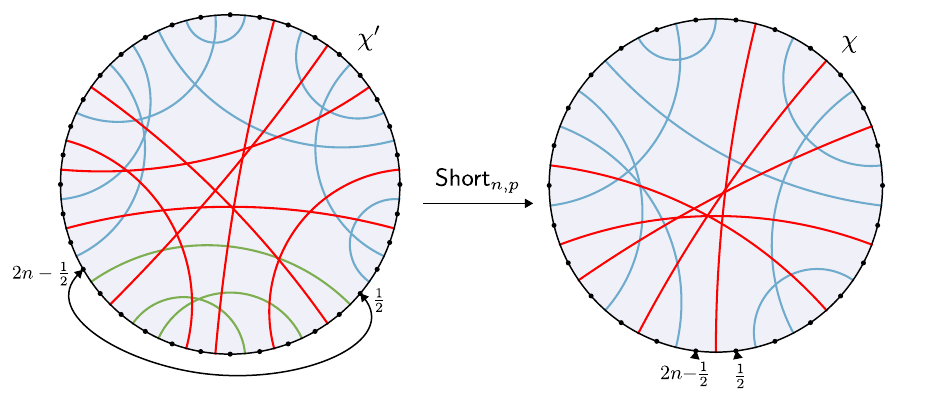}
		\caption{An example of a chord diagram $\chi' \in \mathcal{C}_{2n+2p-2}$ for $n=13$ and $p=5$ and the result $\chi = \mathsf{Short}_{n,p}(\chi')$ of introducing the indicated shortcut. The coloring of the chords will be explained in Figure~\ref{fig:shortcut-decomp} below.}\label{fig:shortcut}
	\end{figure}

We start by establishing several properties of $\mathsf{Short}_{n,p}$.

\begin{lemma}\label{lem:shortplanarity}
	Let $\chi = \mathsf{Short}_{n,p}(\chi')$ and $a,b\in \interval{\tfrac12,2n-\tfrac12}$ with $a < b$. If $d_\chi(a,b) = d_{\chi'}(a,b)$, then $d_\chi(a',b') = d_{\chi'}(a',b')$ for all $a',b' \in \interval{a,b}$.
\end{lemma}
\begin{proof}
	We address this at the level of planar maps, letting $\map$ and $\map'$ be as above. 
	If $d_\chi(a,b) = d_{\chi'}(a,b)$, then there exists a shortest path $\gamma$ between corners $a$ and $b$ in $\map$ that does not pass through the shortcut edge (namely a shortest path in $\map'$).
	Let now $\gamma'$ be a shortest path between $a'$ and $b'$.
	Since $\gamma$ separates $a'$ and $b'$ from the shortcut edge, $\gamma'$ cannot pass through the shortcut edge if it does not cross $\gamma$.
	If $\gamma'$ does cross $\gamma$, then it will have a first and last intersection with $\gamma$, say at vertex $v_1$ and $v_2$ respectively.
	Since $\gamma$ and $\gamma'$ are both shortest, their segments between $v_1$ and $v_2$ are of the same length.
	Therefore replacing this segment of $\gamma'$ by that of $\gamma$ gives another shortest path between $a'$ and $b'$, which does not pass through the shortcut edge.
	In either case, it is a path in $\map'$ as well, so we conclude that $d_\chi(a',b') = d_{\chi'}(a',b')$.
\end{proof}

\begin{lemma}\label{lem:chorddist}
	Let $\chi = \mathsf{Short}_{n,p}(\chi')$. 
	Then $d_\chi(a',b') = d_{\chi'}(a',b')$ for all $a',b' \in \interval{\tfrac12,\chi(2n)-\tfrac12}$ and for all $a',b' \in \interval{\chi(2n)+\tfrac12,2n-\tfrac12}$.
	Moreover, for each chord $\{u < v\}$ of $\chi$, we have that $d_\chi(a',b') = d_{\chi'}(a',b')$ for all $a',b' \in \interval{u+\tfrac12,v-\tfrac12}$.
\end{lemma}
\begin{proof}
	Recall that $\chi(u) = v$ if and only if
	\begin{align*}
		d_\chi(u-\tfrac12,v+\tfrac12) = d_\chi(u+\tfrac12,v-\tfrac12) = d_\chi(u+\tfrac12,v+\tfrac12)-1 = d_\chi(u-\tfrac12,v-\tfrac12)-1.
	\end{align*}
	If $u=2n$, then comparing this to \eqref{eq:shortcutdef} implies that $d_{\chi'}(\tfrac{1}{2},\chi(2n)-\tfrac12) = d_{\chi}(\tfrac{1}{2},\chi(2n)-\tfrac12)$ and $d_{\chi'}(\chi(2n)+\tfrac12,2n-\tfrac12) = d_{\chi}(\chi(2n)+\tfrac12,2n-\tfrac12)$.
	The first claim now follows directly from Lemma~\ref{lem:shortplanarity}.

	So let us assume for the second statement that $u<v<2n$.
	Let $\map$ and $\map'$ be planar maps as before. 
	Let $\gamma_-$ be a shortest path in $\map$ from corner $u-\tfrac12$ to corner $v+\tfrac12$, and $\gamma_+$ a path of equal length from corner $u+\tfrac12$ to corner $v-\tfrac12$.
	These are necessarily disjoint, because if they would meet at a vertex $v$ one could swap segments of the curves to produce a shorter path from $u+\tfrac12$ to $v+\tfrac12$ or from $u-\tfrac12$ to $v-\tfrac12$, in contradiction with the last two equalities in the equation above.
	This means that at most one of $\gamma_-$, $\gamma_+$ can pass through the shortcut edge.
	By planarity this can only be $\gamma_-$, so $d_{\chi}(u+\tfrac12,v-\tfrac12) = d_{\chi'}(u+\tfrac12,v-\tfrac12)$.
	The second claim follows from Lemma~\ref{lem:shortplanarity} again.
\end{proof}

These two lemmas motivate the following classification of chords of a chord diagram $\chi \in \mathcal{C}_{2n}$.
We introduce a partial order on the chords of $\chi$ as follows: for chords $A$ and $B$, we let $A < B$ if $A$ and $B$ do not cross and either $A = \{2n,\chi(2n)\}$ or $A$ separates $B$ from the side $2n$ (i.e.\ $A = \{a ,a'\}$ and $B = \{b, b'\}$ for some $a < b < b' < a' < 2n$).
We thus obtain a partition of the chords into the \emph{minimal} and \emph{nonminimal} ones with respect to this partial order (see the left illustration in Figure~\ref{fig:shortcut-decomp} for an example).

\begin{lemma}\label{lem:nonminimalchords}
	Let $\chi = \mathsf{Short}_{n,p}(\chi')$. Each nonminimal chord of $\chi$ is a chord of $\chi'$ as well.
\end{lemma}
\begin{proof}
	Let $\{u<v\}$ be a nonminimal chord of $\chi$. 
	If $\{u<v\}$ does not cross $\{2n,\chi(2n)\}$, then either $1 \leq u < v \leq \chi(2n)-1$ or $\chi(2n)+1 \leq u < v \leq 2n-1$.
	In both cases, the first statement of Lemma~\ref{lem:chorddist} implies that $d_{\chi}$ and $d_{\chi'}$ agree on the pairwise distances among $u-\tfrac12,u+\tfrac12,v-\tfrac12,v+\tfrac12$. 
	If $\{u < v \}$ crosses $\{2n,\chi(2n)\}$ then there must exist another chord $\{u' < v' \}$ of $\chi$ such that $u' < u <v < v'$.
	The second statement of Lemma~\ref{lem:chorddist} then leads to the same conclusion.
	Therefore, in each case
	\begin{align*}
		d_{\chi'}(u-\tfrac12,v+\tfrac12) = d_{\chi'}(u+\tfrac12,v-\tfrac12) = d_{\chi'}(u+\tfrac12,v+\tfrac12)-1 = d_{\chi'}(u-\tfrac12,v-\tfrac12)-1,
	\end{align*}
	so that $\{u,v\}$ is a chord of $\chi'$.
\end{proof}

The last lemma gives a necessary condition for $\chi'$ so that $\mathsf{Short}_{n,p}(\chi') = \chi$, but not yet a sufficient one.
This is achieved in the following proposition.

\begin{figure}
	\centering
	\includegraphics[width=\linewidth]{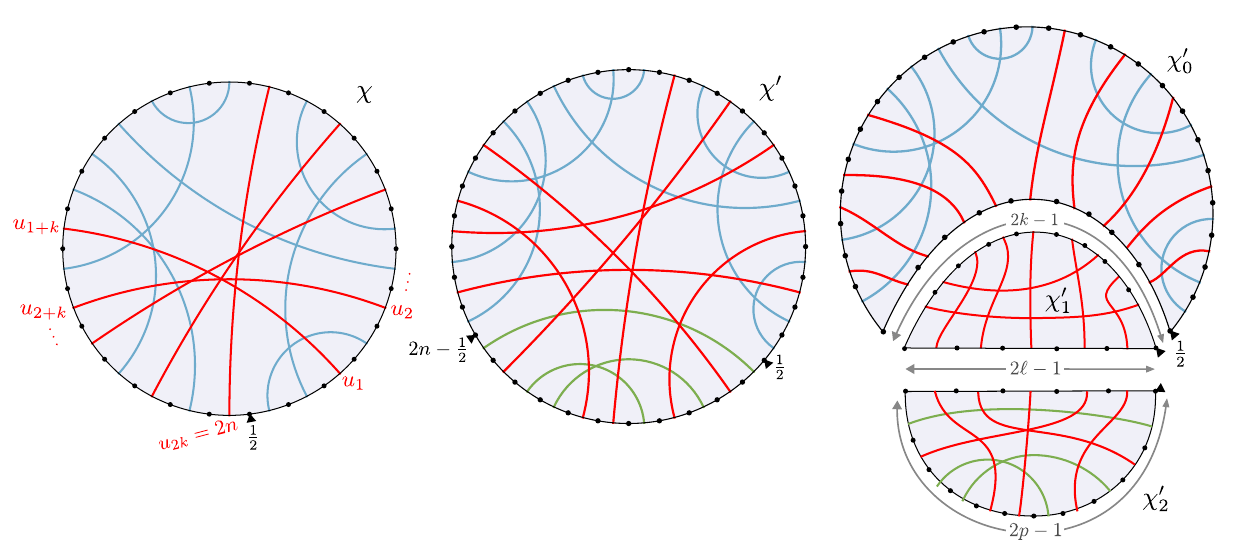}
	\caption{The same example as Figure~\ref{fig:shortcut}. The minimal chords of $\chi$ are shown in red, in this case there are $k=5$ of them, and the nonminimal ones in blue. The active chords of $\chi'$ are shown in red and the chords with both endpoints in $\interval{2n,2n+2p-2}$ in green. On the right the canonical decomposition of $\chi'$ is shown into $\chi_0'$ and $\chi_1'$ and a top part that is completely determined by $\chi$.\label{fig:shortcut-decomp}}
\end{figure}

\begin{proposition}\label{prop:shortsufficient}
	Let $n,p \geq 1$, $\chi \in \mathcal{C}_{2n}$ and $\chi' \in \chords_{2n+2p-2}$. Denote by $k\geq 1$ the number of minimal chords of $\chi$ and let $1\leq u_1 < \cdots < u_{2k} = 2n$ be their endpoints. Then $\mathsf{Short}_{n,p}(\chi') = \chi$ if and only if
	\begin{enumerate}[(i)]
		\item each nonminimal chord of $\chi$ is a chord of $\chi'$ as well;
		\item for each $i = 1, \ldots, k$ the chords of $\chi'$ starting at $u_{i},u_{i+1},\ldots,u_{i+k-1}$ are distinct.
	\end{enumerate}
\end{proposition}
\begin{proof}
Because of Lemma~\ref{lem:nonminimalchords} we may assume that (i) holds for our choice of $\chi$ and $\chi'$.
It is useful to decompose the pseudometric $d_\chi$ as 
\begin{align*}
	d_\chi(a,b) = d_\chi^{\mathrm{min}}(a,b) + d_\chi^{\mathrm{non}}(a,b) \quad \text{for }a,b\in \interval{\tfrac12,2n-\tfrac12},
\end{align*}
where $d_\chi^{\mathrm{min}}$ and $d_\chi^{\mathrm{non}}$ are defined just like \eqref{eq:chorddist} except that only minimal, respectively nonminimal, chords of $\chi$ are taken into account.
Let us call the chords of $\chi'$ that have at least one endpoint at $u_1, \ldots, u_{2k-1}$ the \emph{active} chords of $\chi'$ (the red chords in the middle illustration of Figure~\ref{fig:shortcut-decomp}).
Then we may similarly decompose
\begin{align*}
	d_{\chi'}(a,b) = d_{\chi'}^{\mathrm{act}}(a,b) + d_{\chi}^{\mathrm{non}}(a,b) \quad \text{for }a,b\in \interval{\tfrac12,2n-\tfrac12}
\end{align*}
where $d_{\chi'}^{\mathrm{act}}$ takes into account only active  chords of $\chi'$.
Note that the chords of $\chi'$ starting and ending in $\interval{2n,2n+2p-2}$ do not affect this distance (the green chords in Figure~\ref{fig:shortcut-decomp}).

Note that the minimal chords of $\chi$ each cross pairwise, so $\chi(u_j) = u_{j+k}$ for each $j=1,\ldots,k$.
Therefore
\begin{align}
	d_\chi^{\mathrm{min}}(u_i - \tfrac12,u_{i+k}-\tfrac12) = k\quad \text{for }i=1,\ldots,k.\label{eq:udist}
\end{align}
Recall that $\mathsf{Short}_{n,p}(\chi') = \chi$ is equivalent to
\begin{align}
	d_{\chi}(a,b) = \min\Big[ d_{\chi'}(a,b), d_{\chi'}(a,\tfrac12) + d_{\chi'}(2n-\tfrac12,b) +1 \Big] \quad \text{for all }a<b\in \interval{\tfrac12,2n-\tfrac12}.\label{eq:shortcrit}
\end{align}
We first check that this implies (ii).
From this assumption and \eqref{eq:udist}, it follows that 
\begin{align*}
	d_{\chi'}(u_i - \tfrac12,u_{i+k}-\tfrac12) \geq d_{\chi}(u_i - \tfrac12,u_{i+k}-\tfrac12) = k + d_\chi^{\mathrm{non}}(u_i - \tfrac12,u_{i+k}-\tfrac12).
\end{align*}
Since there are precisely $k$ endpoints of active chords of $\chi'$ in between $u_i - \tfrac12$ and $u_{i+k}-\tfrac12$, namely at $u_i, u_{i+1},\ldots,u_{i+k-1}$, they must correspond to distinct chords to ensure $d_{\chi'}^{\mathrm{act}}(u_i - \tfrac12,u_{i+k}-\tfrac12) = k$.
This proves (ii).

Vice versa, let us assume (ii) holds and deduce \eqref{eq:shortcrit}.
Let $a<b\in\interval{\tfrac12,2n-\tfrac12}$.
Then there exist $i\leq j\in \interval{1,2k}$ such that $u_{i-1} < a < u_i$ and $u_{j-1} < b < u_j$, using the convention $u_0 = 0$.
We claim that 
\begin{align}
	d_{\chi'}(a,b) \geq d_{\chi}(a,b) \quad \text{for }a<b\in\interval{\tfrac12,2n-\tfrac12}\text{ with equality if }j-i \leq k.\label{eq:distineq}
\end{align}
If $j-i \leq k$, assumption (ii) implies that $d_{\chi'}^{\mathrm{act}}(a,b) = j-i = d_{\chi}^{\mathrm{min}}(a,b)$ and therefore $d_{\chi'}(a,b) = d_{\chi}(a,b)$.
If $j-i > k$, then the triangle inequality for $d_{\chi'}^{\mathrm{act}}$ implies that 
\begin{align}
	d_{\chi'}^{\mathrm{act}}(a,b) \geq d_{\chi'}^{\mathrm{act}}(a,u_{i+k}-\tfrac12) - d_{\chi'}^{\mathrm{act}}(u_{i+k}-\tfrac12, b) = k - (j-i-k) = 2k+i-j = d_{\chi}^{\mathrm{min}}(a,b),
\end{align}
so $d_{\chi'}(a,b) \geq d_{\chi}(a,b)$, thereby establishing claim \eqref{eq:distineq}.
It remains to check \eqref{eq:shortcrit} in the following two distinct cases:
\begin{itemize}
	\item If $j-i \leq k$, we just observed that $d_{\chi'}(a,b) = d_{\chi}(a,b)$. By \eqref{eq:distineq}, we also have 
	\begin{align*}
		d_{\chi'}(a,\tfrac12) + d_{\chi'}(2n-\tfrac12,b) +1 \geq d_{\chi}(a,\tfrac12) + d_{\chi}(2n-\tfrac12,b) +1 \geq d_{\chi}(a,b)
	\end{align*}
	where we used the triangle inequality of $d_{\chi}$ in the last step. So \eqref{eq:shortcrit} holds.
	\item If $j-i \geq k+1$, then $i \leq k-1$ and $j \geq k+2$. From \eqref{eq:distineq} it follows that $d_{\chi'}(\tfrac12,a) = d_{\chi}(\tfrac12,a)$ and $d_{\chi'}(b,2n-\tfrac12)=d_{\chi}(b,2n-\tfrac12)$. 
	There is no chord of $\chi$ that starts in $\interval{1,a-\tfrac12}$ and ends in $\interval{b+\tfrac12,2n}$, because such a chord would be smaller, with respect to the partial order, than the minimal chord starting at $u_i$, because $a < u_i < \chi(u_i) = u_{i+k} < b$.  
	Since also $a < u_{k} < b$, the chord $\{2n,\chi(2n)\}$ crosses $\{a,b\}$, it follows that $d_{\chi}(a,\tfrac12) + d_{\chi}(2n-\tfrac12,b) + 1 = d_{\chi}(a,b)$.
	Hence  
	\begin{align*}
		d_{\chi}(a,b) = d_{\chi'}(a,\tfrac12) + d_{\chi'}(2n-\tfrac12,b) + 1. 
	\end{align*}
	Together with $d_{\chi'}(a,b) \geq d_{\chi}(a,b)$, this shows that \eqref{eq:shortcrit} again holds.
\end{itemize}
\end{proof}

\section{Enumeration}\label{sec:enum}

\subsection{Preliminaries on chord diagram enumeration}\label{sec:enumintro}

Let us write $m_{2n}(q)$ for the partition function of chord diagrams of size $2n$,
\begin{align}
	m_{2n}(q) = \sum_{\chi \in \chords_{2n}} q^{\intersec(\chi)}.
\end{align}
By the Touchard-Riordan formula \cite{Touchard1952,Riordan1975}, it is given explicitly by
\begin{align}
	m_{2n}(q) = \frac{1}{(1-q)^n}\sum_{k=0}^n(-1)^k \left[\binom{2n}{n-k}-\binom{2n}{n-k-1}\right] q^{\binom{k+1}{2}}. 
\end{align}
It is well-known \cite{Ismail1987,Penaud1995} that this enumeration is connected with the theory of orthogonal polynomials, particularly the $q$-Hermite polynomials.
By \cite[Sec.~2]{Ismail1987}, the \emph{renormalized $q$-Hermite polynomials} $\widetilde{H}_n(x|q)$ are characterized by the three-term recurrence
\begin{align}
	x \widetilde{H}_n(x|q) = \widetilde{H}_{n+1}(x|q) + [n]_q \widetilde{H}_{n-1}(x|q),\qquad \widetilde{H}_0(x|q) = 1, \quad \widetilde{H}_1(x|q) = x,
\end{align}
where $[n]_q = 1 + q + \cdots + q^{n-1}$ is the $q$-number \eqref{eq:qnumber}.
They are related to the more conventional \emph{continuous $q$-Hermite polynomials} in \eqref{eq:cqhermite} via 
\begin{align}
	\widetilde{H}_n(x|q) = \frac{H_n\left(\frac{\sqrt{1-q}}{2} x \middle| q\right)}{(1-q)^{n/2}}.\label{eq:renormqhermite}
\end{align}
The renormalized $q$-Hermite polynomials are orthogonal with respect to the integral
\begin{align}
	&\left\langle f \right\rangle_q \coloneqq \int_{-\frac{2}{\sqrt{1-q}}}^{\frac{2}{\sqrt{1-q}}} f(x) \rho_q(x)\rmd x, \qquad \rho_q(x) &= \frac{1}{2\pi} \frac{(q;q)_\infty}{\sqrt{\frac{4}{1-q}-x^2}}  \prod_{k=0}^\infty \left(1+(2-x^2(1-q))q^k+q^{2k}\right),\\
	&\left\langle \widetilde{H}_{m}(x|q) \widetilde{H}_{n}(x|q) \right\rangle_q = [n]_q!\,\delta_{mn} \label{eq:qhermiteortho}. 
\end{align}
According to \cite[(3.8)]{Ismail1987}, the partition functions $m_{2n}(q)$ are simply the moments 
\begin{align}
	m_{2n}(q) = \langle x^{2n} \rangle_q.\label{eq:moment}
\end{align}

For a chord diagram $\chi \in\mathcal{C}$ and $1\leq a < b \leq |\chi|$ we say that the circle segment $\interval{a,b}$ is \emph{geodesic} if the chords starting at $a, a+1, \ldots,b$ are all distinct.
The reason for this terminology is that this condition is equivalent to $d_{\chi}(a-\tfrac12,b+\tfrac12) = b-a+1$, so that the path $\interval{a-\tfrac12,b+\tfrac12}$ is a geodesic for the pseudometric $d_{\chi}$.
It follows from \cite[Theorem~3.2]{Ismail1987} for $n_1=n_2=\ldots=n_\ell = 1$ and $n_{\ell+1} = n$ that for $n,\ell\geq 1$ with $\ell+n$ even,
\begin{align}
	\left\langle x^n \widetilde{H}_{\ell}(x|q) \right\rangle_q = \sum_{\substack{\chi \in \chords_{\ell+n}\\ \interval{1,\ell}\text{ geodesic}}} q^{\intersec(\chi)}.\label{eq:distinterpret}
\end{align}
We will also say that $\interval{a,b}$ is \emph{strongly geodesic} if $\interval{a,b}$ is geodesic and the chords starting in $\interval{a,b}$ do not cross each other.
From the orthogonality \eqref{eq:qhermiteortho} and the moment relation \eqref{eq:moment} it follows that for every $n,m\geq 1$ with the same parity,
\begin{align}
	m_{n+m}(q) = \sum_{\ell \geq 0} \left\langle x^n \widetilde{H}_{\ell}(x|q) \right\rangle_q \frac{\left\langle x^m \widetilde{H}_{\ell}(x|q) \right\rangle_q}{[\ell]_q!}.
\end{align}
It is easily seen that this is a manifestation of the canonical decomposition of a chord diagram $\chi \in \mathcal{C}_{n+m}$ into a pair of chord diagrams $\chi_1 \in \mathcal{C}_{n+\ell}$ and $\chi_2\in \mathcal{C}_{m+\ell}$ in which $\interval{1,\ell}$ is geodesic for $\chi_1$ and strongly geodesic for $\chi_2$.
So 
\begin{align}
	\frac{\left\langle x^n \widetilde{H}_{\ell}(x|q) \right\rangle_q}{[\ell]_q!} = \sum_{\substack{\chi \in \chords_{\ell+n}\\ \interval{1,\ell}\text{ strongly geodesic}}} q^{\intersec(\chi)}.\label{eq:strongdistinterpret}
\end{align} 
An important ingredient in the proof of Theorem~\ref{thm:main} is the following refined enumeration of chord diagrams.

\begin{figure}
	\centering
	\includegraphics[width=.5\linewidth]{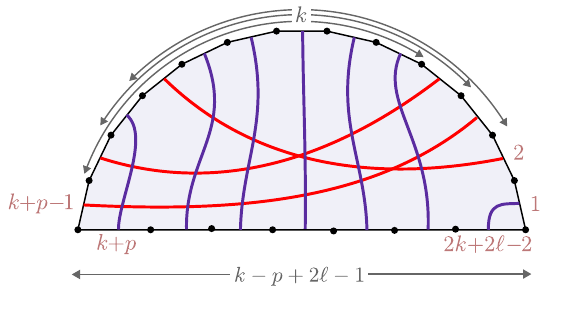}
	\caption{Example of a special chord diagram $\chi \in \mathcal{C}_{k,p,\ell}$ for $k=8$, $p=6$, $\ell = 3$. The bottom segment $k+p,\ldots,2k+2\ell-2$ is strongly geodesic because the $k-p+2\ell-1 = 7$ purple chords are distinct and disjoint. Furthermore, each $k$-tuple of consecutive sides on the top segment belongs to $k$ distinct chords.\label{fig:specialslice}}
\end{figure}

	\begin{proposition}\label{prop:specialslice}
		For $1 \leq \ell \leq p \leq k$, let us denote by $\mathcal{C}_{k,p,\ell} \subset \mathcal{C}_{2k+2\ell-2}$ the chord diagrams of size $2k+2\ell-2$ satisfying the following properties: $\interval{k+p,2k+2\ell-2}$ is strongly geodesic and for each $i=1, \ldots,p$ the segment $\interval{i,i+k-1}$ is geodesic (see Figure~\ref{fig:specialslice}).
		Then we have 
		\begin{align}
			\sum_{\chi \in \mathcal{C}_{k,p,\ell}} q^{\intersec(\chi)} = q^{\binom{k}{2}-\binom{k+\ell-p}{2}} S_q(p,\ell),\label{eq:specialslicegenfun}
		\end{align}
		where $S_q(\cdot,\cdot)$ denotes the $q$-Stirling number of the second kind (see Appendix~\ref{sec:qfunctions}).
	\end{proposition}

	\begin{proof}
		We start by establishing an alternative characterization of these chord diagrams.
		Let $\chi \in \mathcal{C}_{2k+2\ell-2}$.
		We refer to the segment $\interval{1,k+p-1}$ as the \emph{top} of $\chi$ and the complementary segment $\interval{k+p,2k+2\ell-2}$ as the \emph{bottom}. 
		Let us also call a chord \emph{horizontal} if both endpoints are on the top of $\chi$, and \emph{vertical} if it has one endpoint on the top and one on the bottom.
		Then it is not difficult to see that $\chi \in \mathcal{C}_{k,p,\ell}$ if and only if
		\begin{enumerate}[(i)]
			\item all chords of $\chi$ are horizontal or vertical;
			\item all vertical chords are disjoint;
			\item and for each horizontal chord $\{ a,b \}$ we have $|a-b| \geq k$.
		\end{enumerate}
		Note that in this case there are precisely $k-p+2\ell-1$ vertical chords and $p-\ell$ horizontal chords.
		By the last property, each horizontal chord $\{ a,b\}$ with $a<b$ satisfies $1\leq a \leq p-1 < k+1 \leq b \leq k+p-1$.
		From this we deduce that for each $i=p,\ldots,k$ the chord starting at $i$ must be vertical, and must intersect all horizontal chords.
		Hence, there are precisely $i-1 - p+\ell$ vertical chords to its right, so that
		\begin{align}
			\chi(i) = 2k+2\ell-2 - (i-1 -p + \ell) = 2k + \ell + p - i - 1, \quad\text{for }\chi\in \mathcal{C}_{k,p,\ell}\text{ and }i=p,\ldots,k.\label{eq:middlevertical}
		\end{align}

		We will now prove the proposition by fixing $k \geq 1$ and performing induction in $p$. 
		For the base case $p=1$, we necessarily have $\ell=1$ and $\mathcal{C}_{k,1,1}$ contains a single chord diagram with $k$ parallel chords. 
		Hence $\sum_{\chi \in \mathcal{C}_{k,1,1}} q^{\intersec(\chi)}= 1$, which is in agreement with \eqref{eq:specialslicegenfun} because of \eqref{eq:qstirlingk1}.

		Let now $2\leq p \leq k$ and suppose that \eqref{eq:specialslicegenfun} holds with $p$ replaced by $p-1$.
		Let $1\leq \ell \leq p$.
		We partition $\mathcal{C}_{k,p,\ell} = \mathcal{C}^{\mathrm{V}}_{k,p,\ell}\cup \mathcal{C}^{\mathrm{H}}_{k,p,\ell}$ depending on whether the chord at $k+p-1$ is vertical or horizontal, i.e.\ $\mathcal{C}^{\mathrm{V}}_{k,p,\ell} = \{\chi \in \mathcal{C}_{k,p,\ell}: \chi(k+p-1)=k+p\}$.
	
		If $\chi \in \mathcal{C}^{\mathrm{V}}_{k,p,\ell}$, we let $\chi'$ be the chord diagram obtained by deleting the vertical chord $\{k+p-1,k+p\}$.
		This is easily seen to be a bijection $\mathcal{C}^{\mathrm{V}}_{k,p,\ell} \to \mathcal{C}_{k,p-1,\ell-1}$, since the properties (i)-(iii) are preserved.
		Because also $\intersec(\chi) = \intersec(\chi')$, we conclude that
		\begin{align}
			\sum_{\chi \in \mathcal{C}^{\mathrm{V}}_{k,p,\ell}} q^{\intersec(\chi)} = \sum_{\chi \in \mathcal{C}_{k,p-1,\ell-1}} q^{\intersec(\chi')} = q^{\binom{k}{2}-\binom{k+\ell-p}{2}} S_q(p-1,\ell-1).\label{eq:specialsliceA}
		\end{align}

		\begin{figure}
			\centering
			\includegraphics[width=.95\linewidth]{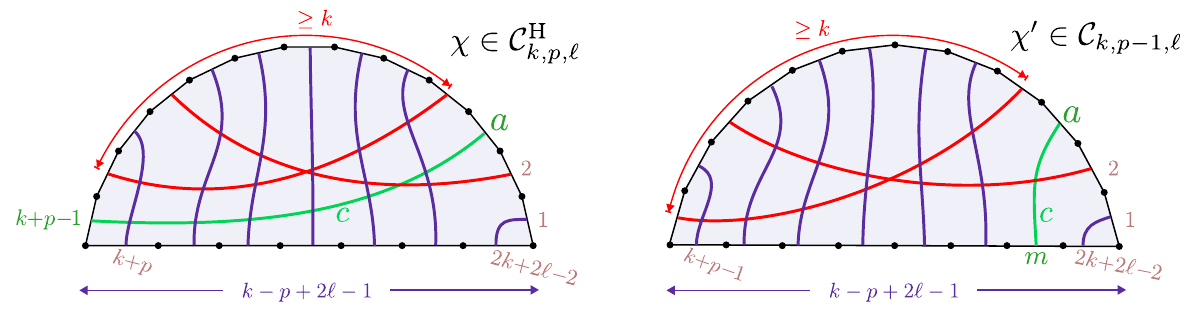}
			\caption{The example $\chi$ of Figure~\ref{fig:specialslice} belongs to $\mathcal{C}^{\mathrm{H}}_{k,p,\ell}$ because the chord $c = \{k+p-1,a\}$ is horizontal (in green). The new chord diagram $\chi'$ is obtained by moving one endpoint of $c$ to become vertical and disjoint from the other vertical chords. In this case $k=8$, $p=6$, $\ell=3$, $a = 3$, $m=19$. \label{fig:specialslice2}}
		\end{figure}
		If $\chi \in\mathcal{C}^{\mathrm{H}}_{k,p,\ell}$, the chord $c = \{k+p-1,a\}$ with $a= \chi(k+p-1)$ is horizontal, and by our previous observations must satisfy $1 \leq a \leq p-1$.
		Let now $\chi'$ be the chord diagram obtained from $\chi$ by turning the horizontal chord $c$ into a vertical chord of $\chi'$ while keeping its endpoint at $a$ (Figure~\ref{fig:specialslice2}).
		To be precise, we move the endpoint of $c$ at $k+p-1$ to the unique position $m > k+p-1$ to ensure that $\chi'$ has a strongly geodesic segment $\interval{k+p-1,2k+2\ell -2}$.
		Since the new chord is disjoint from the other vertical chords, it should be clear that $\chi'$ still obeys properties (i)-(iii) if we take the top of $\chi'$ to be shorter by one, thus $\chi' \in \mathcal{C}_{k,p-1,\ell}$.
		Given the pair $(m,\chi')$, one can of course reconstruct $\chi$ by moving the endpoint at $m$ back to $k+p-1$.
		Let us denote this moving operation for arbitrary $m=k+p,\ldots,2k+2\ell-2$ by $\mathsf{Move}_{k,p,\ell}(m,\chi')$.
		We claim that restricted to $m \geq 2k+\ell-1$ this determines a bijection
		\begin{align}
			\mathsf{Move}_{k,p,\ell}:\interval{2k+\ell-1,2k+2\ell-2} \times \mathcal{C}_{k,p-1,\ell} \to \mathcal{C}^{\mathrm{H}}_{k,p,\ell}.\label{eq:movebijection}
		\end{align}

		To verify the bijection, it is sufficient to check that for $k+p \leq m \leq 2k+2\ell-2$ and $\chi'\in\mathcal{C}_{k,p-1,\ell}$, the result $\chi=\mathsf{Move}_{k,p,\ell}(m,\chi')\in \mathcal{C}^{\mathrm{H}}_{k,p,\ell}$ if and only if $m \geq 2k+\ell-1$.
		Apart from the special chord $c$, the horizontal and vertical chords of $\chi'$ and their properties (i)-(iii) are preserved by $\mathsf{Move}_{k,p,\ell}$, so this is equivalent to $c$ itself satisfying (iii), i.e.\ $\chi(k+p-1) \leq p-1$.
		By \eqref{eq:middlevertical} with $p$ replaced by $p-1$, we have that $\chi'(2k+\ell-1) = p-1$ and $\chi'(2k+\ell-2) = p$.
		Since vertical chords are disjoint, this implies that $\chi'(m) \leq p-1$ if and only if $m \geq 2k+\ell-1$.
		Since $\chi(k+p-1) = \chi'(m)$, this verifies the claimed bijection.

		If $\chi = \mathsf{Move}_{k,p,\ell}(m,\chi')$, we can easily compare the number of crossings of the special chord $c$ in $\chi'$ to that in $\chi$.
		The number of crossings with horizontal chords is unchanged, while $c$ crosses exactly $m-k-p+1$ vertical chords in $\chi$ and none in $\chi'$ (see again Figure~\ref{fig:specialslice2}).
		Hence, $\ell_\chi(c) = \ell_{\chi'}(c) + m-k-p+1$.
		Since the other chords are unaffected, $\intersec(\chi) = m-k-p+1 + \intersec(\chi')$.
		The bijection \eqref{eq:movebijection} then implies
		\begin{align*}
		\sum_{\chi \in \mathcal{C}^{\mathrm{H}}_{k,p,\ell}} q^{\intersec(\chi)} &= \sum_{m=2k+\ell-1}^{2k+2\ell-2} q^{m-k-p+1} \sum_{\chi \in \mathcal{C}_{k,p-1,\ell}} q^{\intersec(\chi')}  \\
		&= q^{k+\ell-p} [\ell]_q\,q^{\binom{k}{2}-\binom{k+\ell-p+1}{2}} S_q(p-1,\ell)\\
		&= [\ell]_q\,q^{\binom{k}{2}-\binom{k+\ell-p}{2}} S_q(p-1,\ell).
		\end{align*}
		Combining with \eqref{eq:specialsliceA} and using the recurrence relation \eqref{eq:qstirlingsecond}, this gives
		\begin{align*}
			\sum_{\chi \in \mathcal{C}_{k,p,\ell}} q^{\intersec(\chi)} = q^{\binom{k}{2}-\binom{k+\ell-p}{2}} \left( S_q(p-1,\ell-1) + [\ell]_q S_q(p-1,\ell) \right) = q^{\binom{k}{2}-\binom{k+\ell-p}{2}} S_q(p,\ell).
		\end{align*} 
		This completes the proof by induction.
	\end{proof}

\subsection{What is special about the weights?}

Recall that we will use the special weight sequence $\mathbf{t}(q)$ defined in \eqref{eq:specweights} whose potential derivative is given by 
\begin{align}
	V_q'(x) &= x - \sum_{k=1}^\infty t_{2k}(q) x^{2k-1}\nonumber\\
	&= 2\sqrt{1-q} \sum_{\ell\geq 1} (-1)^{\ell-1} q^{\binom{\ell}{2}} T_{2\ell-1}\left(\frac{\sqrt{1-q}}{2} x\right)\label{eq:potential}
\end{align}
in terms of the Chebyshev polynomials of the first kind.
Besides the fact that $t_{2k}(q) = O(q)$ for all $k\geq 1$, the only property we will use about these weights in the proog of Theorem~\ref{thm:main} is the following expression for the inner product between the potential derivative and the renormalized $q$-Hermite polynomials.

\begin{lemma}\label{lem:potentialinqhermite}
	For each $\ell \geq 1$, the potential derivative $V_q'(x)$ of \eqref{eq:potential} and the odd-degree renormalized Hermite polynomial of \eqref{eq:renormqhermite} satisfy
	\begin{align*}
		\left\langle V_q'(x)\widetilde{H}_{2\ell-1}(x|q) \right\rangle_q = q^{\binom{\ell}{2}} s_q(\ell,1),
	\end{align*}
	where $s_q(\cdot,\cdot)$ denotes the $q$-Stirling number of the first kind (see Appendix~\ref{sec:qfunctions}).
\end{lemma}
\begin{proof}
	From the expansion \eqref{eq:qhermitechebyshev} of the continuous $q$-Hermite polynomials in terms of Chebyshev polynomials and the relation \eqref{eq:renormqhermite} it follows that
	\begin{align}
		\widetilde{H}_{2k-1}(x|q) = 2(1-q)^{\frac12-k} \sum_{p=1}^k \qbinom{2k-1}{k-p}_q T_{2p-1}\left(\frac{\sqrt{1-q}}{2} x\right).
	\end{align}
	We claim that
	\begin{align}
		V_q'(x) = \sum_{k\geq 1} (-1)^{k+1} q^{\binom{k}{2}} \frac{[k-1]_q!}{[2k-1]_q!} \widetilde{H}_{2k-1}(x|q).
	\end{align}
	In that case, the orthogonality \eqref{eq:qhermiteortho} of the $q$-Hermite polynomials implies that for $\ell\geq 1$ we have the identity
	\begin{align*}
		\left\langle V_q'(x)\widetilde{H}_{2\ell-1}(x|q) \right\rangle_q = q^{\binom{\ell}{2}} (-1)^{\ell+1} [\ell-1]_q! \stackrel{\eqref{eq:qstirlingk1}}{=} q^{\binom{\ell}{2}} s_q(\ell,1).
	\end{align*}
	To establish the claim it is sufficient to show that
	\begin{align}
		\sum_{k=p}^\infty (-1)^{k-p} q^{\binom{k}{2}-\binom{p}{2}} (1-q)^{-k} \frac{[k-1]_q!}{[2k-1]_q!}\qbinom{2k-1}{k-p}_q = 1
	\end{align}
	for each $p\geq 1$.
	To this end, we rewrite the left-hand side with the help of \eqref{eq:qbinom} and \eqref{eq:qfactorialpoch} as
	\begin{align*}
		\sum_{k=p}^\infty (-1)^{k-p} q^{\binom{k}{2}-\binom{p}{2}} \frac{(q;q)_{k-1}}{(q;q)_{k-p}(q;q)_{k+p-1}}
		&= \sum_{\ell=0}^\infty (-1)^{\ell} \frac{q^{\binom{\ell}{2}}}{(q;q)_{\ell}} q^{p\ell}  \frac{(q;q)_{p+\ell-1}}{(q;q)_{2p+\ell-1}}\\
		&= \sum_{\ell=0}^\infty \frac{q^{\binom{\ell}{2}}}{(q;q)_{\ell}} \frac{(-q^p)^\ell}{(q^{p+\ell};q)_{p}}\\
		&\stackrel{\eqref{eq:qbinomtheorem}}{=} \sum_{\ell=0}^\infty \frac{q^{\binom{\ell}{2}}}{(q;q)_{\ell}} (-q^p)^\ell \sum_{k=0}^\infty \qbinom{p+k-1}{k}_q q^{(p+\ell)k}\\
		&= \sum_{k=0}^\infty\qbinom{p+k-1}{k}_q q^{pk}\sum_{\ell=0}^\infty \frac{q^{\binom{\ell}{2}}}{(q;q)_{\ell}} (-q^{p+k})^\ell\\
		&\stackrel{\eqref{eq:qinftyseries}}{=} \sum_{k=0}^\infty\qbinom{p+k-1}{k}_q q^{pk} (q^{p+k};q)_\infty\\
		&\stackrel{\eqref{eq:qbinom}}{=} \sum_{k=0}^\infty \frac{q^{pk}}{(q;q)_{k}} (q^p;q)_k(q^{p+k};q)_\infty \\
		&= \sum_{k=0}^\infty \frac{q^{pk}}{(q;q)_{k}} (q^p;q)_\infty \stackrel{\eqref{eq:reciprocalqpoch}}{=} 1.
	\end{align*}
	This finishes the proof of the claimed expansion.

\end{proof}

\subsection{Main chord diagram identity}

We are ready to combine Proposition~\ref{prop:shortsufficient}, which characterizes the chord diagrams $\chi'$ that yield the same chord diagram $\chi$ after introducing a shortcut, with the enumeration identities Proposition~\ref{prop:specialslice} and Lemma~\ref{lem:potentialinqhermite}.

\begin{proposition}\label{prop:tuttechord}
	Let $n\geq1$ and $\chi \in \mathcal{C}_{2n}$. Then 
	\begin{align}
		\sum_{p\geq 1} t_{2p}(q) \sum_{\substack{\chi' \in \chords_{2n+2p-2}\\ \mathsf{Short}_{n,p}(\chi')=\chi}} q^{\intersec(\chi')} = \begin{cases} q^{\intersec(\chi)} &\text{if }\ell_\chi(2n) > 0\\ 0 &\text{if }\ell_\chi(2n)=0,\end{cases}\label{eq:chordclaim}
	\end{align}
	where $\ell_\chi(2n)$ is the number of chords crossing the chord starting at $2n$.
\end{proposition}
\begin{proof}
	Let us fix $n,p\geq 1$ and $\chi\in \mathcal{C}_{2n}$ and denote the number of minimal chords of $\chi$ by $k$.
	If $\chi' \in \mathcal{C}_{2n+2p-2}$, then $d_{\chi'}(\tfrac12,2n-\tfrac12) = 2\ell-1 \geq 1$ is necessarily odd.
	We observe that $\chi'$ satisfies properties (i) and (ii) of Proposition~\ref{prop:shortsufficient} precisely when we can decompose $\chi'$ into a triple of chord diagrams $\chi_0',\chi_1',\chi_2'$ (see the right illustration of Figure~\ref{fig:shortcut-decomp}): 
	\begin{itemize}
		\item the top part $\chi_0' \in \mathcal{C}_{2n+2k-2}$ is completely determined by $\chi$ and contains all nonminimal chords of $\chi$ as well as $2k-1$ disjoint chords $\chi_0'(u_i)=2k-1-i$;
		\item the middle part $\chi_1' \in \mathcal{C}_{2k+2\ell-2}$ is such that the chords starting in $\interval{2k,2k+2\ell-2}$ are all distinct and disjoint, while the chords at $i,i+1,\ldots,i+k-1$ are distinct for all $i=1,\ldots,k$;
		\item the bottom part $\chi_2' \in \mathcal{C}_{2\ell+2p-2}$ in which the chords starting at $\interval{1,2\ell-1}$ are distinct.
	\end{itemize}
	We note that $\chi_1'$ is precisely of the type $\mathcal{C}_{k,k,\ell}$ enumerated in Proposition~\ref{prop:specialslice}, while $\chi_2'$ is precisely of the type enumerated by $\langle \widetilde{H}_{2\ell-1}(x|q)x^{2p-1}\rangle_q$ in \eqref{eq:distinterpret}.
	Since 
	\begin{align*}
	\intersec(\chi') = \intersec(\chi_0') + \intersec(\chi_1')+\intersec(\chi_2') = \intersec(\chi) - \binom{k}{2} + \intersec(\chi_1')+\intersec(\chi_2'),
	\end{align*}
	we deduce from Proposition~\ref{prop:shortsufficient}, Proposition~\ref{prop:specialslice} and \eqref{eq:distinterpret} that
	\begin{align*}
		\sum_{\substack{\chi' \in \chords_{2n+2p-2}\\ \mathsf{Short}_{n,p}(\chi')=\chi}} q^{\intersec(\chi')} &= q^{\intersec(\chi)-\binom{k}{2}}\sum_{\ell=1}^k\sum_{\chi_1' \in \chords_{k,k,\ell}} q^{\intersec(\chi_1')}\sum_{\substack{\chi_2' \in \chords_{2\ell+2p-2}\\\interval{1,2\ell-1}\text{ strongly geodesic}}} q^{\intersec(\chi_2')} \\
		&= q^{\intersec(\chi)}\sum_{\ell=1}^k  S_q(k,\ell) q^{-\binom{\ell}{2}}\langle \widetilde{H}_{2\ell-1}(x|q)x^{2p-1}\rangle_q.
	\end{align*}
	Applying Lemma~\ref{lem:potentialinqhermite} gives
	\begin{align}
		\sum_{p\geq 1} t_{2p}(q) \sum_{\substack{\chi' \in \chords_{2n+2p-2}\\ \mathsf{Short}_{n,p}(\chi')=\chi}} q^{\intersec(\chi')} &= q^{\intersec(\chi)} \sum_{\ell \geq 1} S_q(k,\ell) q^{-\binom{\ell}{2}}\left\langle (x - V_q'(x))\widetilde{H}_{2\ell-1}(x|q)\right\rangle_q\nonumber\\
		&= q^{\intersec(\chi)} \left(S_q(k,1) - \sum_{\ell \geq 1} S_q(k,\ell) s_q(\ell,1)\right)\\
		&= q^{\intersec(\chi)} (1 - \delta_{k,1}),\nonumber
	\end{align}
	where in the last step we used \eqref{eq:qstirlingk1} and  the orthogonality relation \eqref{eq:qstirlinginverse}.
	Since the only way to have just one minimal chord, $k=1$, is if the chord $\{2n,\chi(2n)\}$ is not crossed at all, $\ell_\chi(2n)=0$, this finishes the proof.
\end{proof}

\subsection{Proofs of Theorem~\ref{thm:main} and Corollary~\ref{cor:probabilistic}}\label{sec:mainproof}

\begin{proof}[Proof of Theorem~\ref{thm:main}]
	We will prove the theorem by induction on the order $m$ of $q$ and the size $2n$ of $\chi$ simultaneously.
	For the base case $m=0$ and $n\geq 1$, since $t_{2k}(q) = O(q)$ for all $k\geq 1$, the only maps contributing non-trivially to $F(\mathbf{t}(q);\chi)$ at zeroth order in $q$ are the ones without inner faces, i.e.\ plane trees. Such a map only exists when $\intersec(\chi) = 0$ and is uniquely given by gluing the sides of the $2n$-gon according to the involution $\chi$. Hence $F(\mathbf{t}(q);\chi) = q^{\intersec(\chi)} + O(q)$ holds for chord diagrams of arbitrary size.

	Let $m,n \geq 1$ and suppose that $F(\mathbf{t}(q);\chi) = q^{\intersec(\chi)} + O(q^m)$ for chord diagrams $\chi$ of arbitrary size and $F(\mathbf{t}(q);\chi) = q^{\intersec(\chi)} + O(q^{m+1})$ for chord diagrams of size less than $2n$.
	Let us show that the last statement holds for each chord diagram of size $2n$.
	To this end we fix a $\chi \in \chords_{2n}$.
	We need to consider the maps $\map\in\maps_{2n}$ such that $\mathsf{Geod}(\map) = \chi$.
	Let us examine what happens if we delete from $\map$ the edge labelled $2n$, immediately to the left of the root corner.
	
	Suppose $\ell_\chi(2n) > 0$.
	Then after deletion the map is still connected and thus corresponds to a bipartite planar map $\map'\in\maps_{2n+2p-2}$ for some $p \geq 1$.
	We can reconstruct $\map$ uniquely from $\map'$ by drawing a new edge in the outer face of $\map'$ from corner $\tfrac12$ to $2n-\tfrac12$.
	In this case
	\begin{align*}
		\mathsf{Geod}(\map) = \mathsf{Short}_{n,p}(\mathsf{Geod}(\map'))
	\end{align*}
	and $w_{\mathbf{t}(q)}(\map) = t_{2p}(q) w_{\mathbf{t}(q)}(\map')$.
	Therefore
	\begin{align*}
		F(\mathbf{t}(q); \chi) &= \sum_{\map \in \maps_{\chi}} w_{\mathbf{t}(q)}(\map)= \sum_{p\geq 1} t_{2p}(q) \sum_{\substack{\chi' \in \mathcal{C}_{2n+2p-2}\\\mathsf{Short}_{n,p}(\chi')=\chi}}\sum_{\map' \in \maps_{\chi'}} w_{\mathbf{t}(q)}(\map')\\
		&=\sum_{p\geq 1} t_{2p}(q) \sum_{\substack{\chi' \in \mathcal{C}_{2n+2p-2}\\\mathsf{Short}_{n,p}(\chi')=\chi}}F(\mathbf{t}(q); \chi')=\sum_{p\geq 1} t_{2p}(q) \sum_{\substack{\chi' \in \mathcal{C}_{2n+2p-2}\\\mathsf{Short}_{n,p}(\chi')=\chi}} q^{\intersec(\chi')} + O(q^{m+1}),
	\end{align*}
	where we used that $t_{2p}(q) = O(q)$ in the last step.
	Then Proposition~\ref{prop:tuttechord} gives the claimed identity.

	Suppose instead that $\ell_\chi(2n)=0$.
	Because the right-hand side of \eqref{eq:chordclaim} vanishes, the same argument applies to show that the maps $\map$ that do not disconnect after edge deletion contribute $O(q^{m+1})$ to $F(\mathbf{t}(q);\chi)$.
	It remains to identify the contribution of maps $\map$ that decompose into a pair of bipartite planar maps $\map_1$ and $\map_2$, where we allow $\map_1$ and/or $\map_2$ to consist of a single vertex (zero perimeter).
	We must have that $\mathsf{Geod}(\map_1)=\chi_1$ and $\mathsf{Geod}(\map_2)=\chi_2$ where $\chi_1$ and $\chi_2$ are the, possibly empty, subdiagrams of $\chi$ sitting to the left and to the right of the chord $\{2n,\chi(2n)\}$.
	Since $w_{\mathbf{t}(q)}(\map) = w_{\mathbf{t}(q)}(\map_1)w_{\mathbf{t}(q)}(\map_2)$, we have
	\begin{align*}
		F(\mathbf{t}(q); \chi) &= \sum_{\map_1 \in \maps_{\chi_1}} w_{\mathbf{t}(q)}(\map_1)\sum_{\map_2 \in \maps_{\chi_2}} w_{\mathbf{t}(q)}(\map_2) + O(q^{m+1}) \\
		&= F(\mathbf{t}(q); \chi_1)F(\mathbf{t}(q); \chi_2)+ O(q^{m+1}).
	\end{align*}
	But $\chi_1$ and $\chi_2$ both have size less than $2n$ and $\intersec(\chi) = \intersec(\chi_1) + \intersec(\chi_2)$, so
	\begin{align*}
		F(\mathbf{t}(q); \chi) &= q^{\intersec(\chi_1)}q^{\intersec(\chi_2)} + O(q^{m+1}) = q^{\intersec(\chi)} + O(q^{m+1}).
	\end{align*}
	This finishes the proof by induction.		
\end{proof}

Finally, we provide the details for the probabilistic interpretation.

\begin{proof}[Proof of Corollary~\ref{cor:probabilistic}]
	Recall from the discussion in the introduction that the sum $\sum_{\map\in\maps_\chi} |w_{\mathbf{t}(q)}(\map)| = F(|\mathbf{t}(q)|;\chi)$ converges if and only if $F_{2n}(|\mathbf{t}(q)|) < \infty$ for $n \geq 1$, i.e.\ when $|\mathbf{t}(q)|$ is an admissible weight sequence.
	Recall also from \eqref{eq:admissible} that this happens precisely when 
	\begin{align}
		g_q(r) \coloneqq r - \sum_{k=1}^\infty |t_{2k}(q)|\, \binom{2k-1}{k} r^k = 1\label{eq:admcondition}
	\end{align}
	has a positive solution.

	From here on we assume $0<q<1$.
	First we observe that 
	\begin{align}
		-t_2(q) &= -1 + (1-q) \sum_{\ell\geq 1} (2\ell-1) q^{\binom{\ell}{2}} \nonumber \\
		&\geq -1 + (1-q)\left(-\frac{1}{2} + \frac{3}{2}\sum_{\ell\geq 1} q^{\binom{\ell}{2}} \sum_{p=0}^{\ell-1} q^p\right)\nonumber\\
		&= -1 + (1-q)\left(-\frac{1}{2} + \frac{3}{2}\frac{1}{1-q}\right) = \frac{q}{2}.\label{eq:t2lowerbound}
	\end{align}
	In particular, since this is positive, we have for all $k\geq 1$ that
	\begin{align}
		|t_{2k}(q)| = -\delta_{k,1} + (1-q)^k\sum_{\ell \geq k} q^{\binom{\ell}{2}} \frac{2\ell-1}{2k-1}\binom{\ell+k-2}{\ell-k}.
	\end{align}
	Using the crude upper bound $q^{\binom{\ell}{2}} \leq q^{\ell-1}$ leads to
	\begin{align}
		|t_{2k}(q)| \leq -\delta_{k,1} + (1-q)^k\sum_{\ell \geq k} q^{\ell-1} \frac{2\ell-1}{2k-1}\binom{\ell+k-2}{\ell-k} = -\delta_{k,1} + \frac{1+q}{q} \left(\frac{q}{1-q}\right)^{k}.
	\end{align}
	It follows that the left-hand side of \eqref{eq:admcondition} has radius of convergence at least $\frac{1-q}{4q}$, and that for $0 \leq r < \frac{1-q}{4q}$ we have the lower bound 
	\begin{align}
		g_q(r) &\geq 2r - \frac{1+q}{q}\sum_{k=1}^\infty  \left(\frac{q r}{1-q}\right)^{k}\, \binom{2k-1}{k}= 2r - \frac{1+q}{2q} \left( \frac{1}{\sqrt{1 - 4 \frac{qr}{1-q}}}-1\right).
	\end{align}
	In particular, for $r = \frac{1-q}{12 q}$ the right-hand side equals $(8+4q-\sqrt{54}(1+q))/(12q)$, which exceeds $1$ for $q \leq (8-\sqrt{54})/(8+\sqrt{54})=0.0424\ldots$.
	Since $g_q$ is analytic and $g_q(0) = 0$, \eqref{eq:admcondition} has a positive solution for each $q \in (0,0.042)$, proving admissibility in this regime.

	Using that $(2\ell-1)\binom{\ell}{2} \geq \frac{1}{2}(\ell+1)\ell(\ell-1) = \sum_{p=0}^{\ell-1} (\binom{\ell}{2}+p)$ for all $\ell \geq 1$, we find the lower bound
	\begin{align*}
		|t_4(q)| &= \frac{(1-q)^2}{3} \sum_{\ell \geq 2} (2\ell-1)\binom{\ell}{2} q^{\binom{\ell}{2}} \\
		&\geq \frac{(1-q)^2}{3} \sum_{\ell \geq 2}q^{\binom{\ell}{2}} \sum_{p=0}^{\ell-1} \left(\binom{\ell}{2}+p\right) q^p = \frac{(1-q)^2}{3} \sum_{p=0}^\infty p q^p = \frac{q}{3}.
	\end{align*}
	Combining with \eqref{eq:t2lowerbound}, we find for all $r\geq 0$ that
	\begin{align*}
		g_q(r) \leq r - |t_2(q)| r - 3 |t_4(q)| r^2 \leq r - \frac{q}{2} r -q r^2 \leq \frac{(2-q)^2}{16 q},
	\end{align*}
	where the last inequality follows from evaluating at the maximum at $r = (2-q)/(4q)$. 
	Hence \eqref{eq:admcondition} has no positive solution when $(2-q)^2 < 16 q$, so when $q > 10 - \sqrt{96} = 0.20204\ldots$.
	This proves $|\mathbf{t}(q)|$ is not admissible for $q > 0.21$.

	Now let $|\mathbf{t}(q)|$ be admissible and let $M_{2n}$ be the $|\mathbf{t}(q)|$-Boltzmann planar map.
	By definition of the latter, it holds for all $n,v \geq 1$ and $\chi \in \mathcal{C}_{2n}$ that
	\begin{align*}
		\mathbb{P}(\mathsf{Geod}(M_{2n}) = \chi,\mathsf{V}(M_{2n})=v) &= \frac{1}{F_{2n}(|\mathbf{t}(q)|)} \sum_{\map\in\maps_{\chi}} w_{|\mathbf{t}(q)|}(\map)\, \ind_{\{\mathsf{V}(\map)=v\}}\\
		&= \frac{(-1)^{v-n-1}}{F_{2n}(|\mathbf{t}(q)|)} \sum_{\map\in\maps_{\chi}}w_{\mathbf{t}(q)}(\map)\ind_{\{\mathsf{V}(\map)=v\}}.
	\end{align*}
	Therefore, multiplying by $(-1)^v$ and summing over $v$ we conclude that
	\begin{align*}
		&\mathbb{P}(\mathsf{Geod}(M_{2n}) = \chi, \,\mathsf{V}(M_{2n})\text{ even}) \,\,-\,\, \mathbb{P}(\mathsf{Geod}(M_{2n}) = \chi,\,\mathsf{V}(M_{2n})\text{ odd})\\
		&\quad = \frac{(-1)^{n-1}}{F_{2n}(|\mathbf{t}(q)|)} \sum_{\map\in\maps_{\chi}}w_{\mathbf{t}(q)}(\map)= \frac{(-1)^{n-1}}{F_{2n}(|\mathbf{t}(q)|)} F(\mathbf{t}(q);\chi) = \frac{(-1)^{n-1}}{F_{2n}(|\mathbf{t}(q)|)}\, q^{\intersec(\chi)},
	\end{align*}
	where we used the absolute convergence of the infinite sum and the last equality follows from Theorem~\ref{thm:main}.	 
\end{proof}
		
\appendix
\section{$q$-functions and identities}\label{sec:qfunctions}
The following formulas are used in this work, see for instance \cite{Gasper1990,Koekoek1996,Koekoek2010,Ernst2012} for an overview.
The \emph{$q$-number} and \emph{$q$-factorial} are 
\begin{align}
	[n]_q = \frac{1-q^n}{1-q}, \quad [n]_q! = \prod_{k=1}^n [k]_q.\label{eq:qnumber}
\end{align}
The \emph{$q$-Pochhammer symbol}
\begin{align}
	(a;q)_n = \prod_{k=0}^{n-1}(1-a q^k)
\end{align}
satisfies the identities
\begin{align}
	(a;q)_n &= \frac{(a;q)_\infty}{(a q^n;q)_\infty},\label{eq:qpochinfty}\\
	(q;q)_n &= (1-q)^n [n]_q!,\label{eq:qfactorialpoch}\\
	(a;q)_\infty &= \sum_{k\geq 0}q^{\binom{k}{2}} \frac{(-a)^k}{(q;q)_k}\label{eq:qinftyseries}\\
	\frac{1}{(a;q)_\infty} &= \sum_{k\geq 0} \frac{a^k}{(q;q)_k}\label{eq:reciprocalqpoch}.
\end{align}
The \emph{$q$-binomial} is defined as
\begin{align}
	\qbinom{n}{m}_q = \frac{[n]_q!}{[m]_q![n-m]_q!} = \frac{(q^{n-m+1};q)_m}{(q;q)_m},\label{eq:qbinom}
\end{align}
which satisfies the $q$-binomial theorems 
\begin{align}
	(a;q)_n = \sum_{k=0}^n q^{\binom{k}{2}}\qbinom{n}{k}_q (-a)^k,\quad \frac{1}{(a;q)_n} = \sum_{k=0}^\infty \qbinom{n+k-1}{k}_q a^k.\label{eq:qbinomtheorem}
\end{align}

The \emph{continuous $q$-Hermite polynomials} $H_n(x|q)$ (see e.g.\ \cite[Sec.~3.26]{Koekoek1996}) are uniquely determined by the three-term recurrence
\begin{align}
	2x H_n(x|q) = H_{n+1}(x|q) + (1-q^n) H_{n-1}(x|q), \quad H_0(x|q) = 1, \quad H_1(x|q) = 2x.\label{eq:cqhermite}
\end{align}
They can be written rather explicitly as 
\begin{align}
	H_n(x|q) = \sum_{k=0}^n \qbinom{n}{k}_q e^{i(n-2k)\theta}, \qquad x = \cos \theta.
\end{align}
In particular, for $n = 2\ell-1$ odd we have
\begin{align}
	H_{2\ell-1}(x|q) = 2\sum_{p=1}^\ell \qbinom{2\ell-1}{\ell-p}_q T_{2p-1}\left(x\right),\label{eq:qhermitechebyshev}
\end{align}
in terms of the Chebyshev polynomials $T_n(x)$.

We will also encounter the \emph{$q$-Stirling numbers  of the first kind} $s_q(n,k)$ and of the \emph{second kind} $S_q(n,k)$ (see \cite{Gould1961} or \cite[Ch.~5]{Ernst2012}). 
They are uniquely determined by the recurrence relations
\begin{align}
	s_q(n+1,k) &= s_q(n,k-1) - [n]_q\, s_q(n,k), \quad s_q(n,0) = \delta_{n,0}, \quad s_q(0,k) = \delta_{0,k}.\label{eq:qstirlingfirst}\\
	S_q(n+1,k) &= S_q(n,k-1) + [k]_q\, S_q(n,k), \quad S_q(n,0) = \delta_{n,0}, \quad S_q(0,k) = \delta_{0,k}.\label{eq:qstirlingsecond}
\end{align}
In particular,
\begin{align}
	s_q(n,1) = (-1)^{n-1} [n-1]_q!\quad \text{and}\quad S_q(n,1) = 1\quad \text{for }n\geq 1.\label{eq:qstirlingk1}
\end{align}
They may be viewed as matrix elements of a pair of infinite lower unitriangular matrices because $s_q(n,n)=S_q(n,n)=1$ and $s_q(n,k) = S_q(n,k) = 0$ for $n < k$, and these matrices are inverses of each other \cite[Thm.~5.2.7]{Ernst2012},
\begin{align}
	\sum_{j=k}^n S_q(n,j)s_q(j,k) = \sum_{j=k}^n s_q(n,j)S_q(j,k) = \delta_{n,k}.\label{eq:qstirlinginverse}
\end{align}
\bibliographystyle{siam}
\bibliography{ref}	
	
\end{document}